\documentclass[11pt]{amsart}
\usepackage[utf8x]{inputenc}
\usepackage{amsfonts}
\usepackage{amssymb}
\usepackage{amsmath}
\usepackage{delarray}
\usepackage{graphicx}
\usepackage[all]{xy}
\newtheorem{thm}{Theorem}[section]

\newtheorem{prop}[thm]{Proposition}
\newtheorem{conj}[thm]{Conjecture}

\newtheorem{defn}[thm]{Definition}
\theoremstyle{remark}
\newtheorem{remark}[thm]{Remark}
\numberwithin{equation}{section}
\newtheorem{example}[thm]{Example}

\begin{document}

\newpage
\title{Reduction and integrability \\ of stochastic dynamical systems}

\author{Nguyen Tien Zung \& Nguyen Thanh Thien}
\address{Institut de Mathématiques de Toulouse, UMR5219, Université Toulouse 3}
\email{tienzung.nguyen@math.univ-toulouse.fr, ttnguyen@math.univ-toulouse.fr.}


\dedicatory{Dedicated to Anatoly Timofeevich Fomenko on the occasion of his $70th$ birthday}
\begin{abstract} 
This paper is devoted to the study of qualitative geometrical properties of stochastic dynamical systems, namely their symmetries, reduction and integrability. In particular, we show that an SDS which is diffusion-wise symmetric with respect to a proper Lie group action can be diffusion-wise reduced to an SDS on the quotient space. We also show necessary and sufficient conditions for an SDS to be projectable via a surjective map. We then introduce the notion of integrability of SDS's, and extend the results about the existence and structure-preserving property of Liouville torus actions from the classical case to the case of integrable SDS's. We also show how integrable SDS's are related to compatible families of integrable Riemannian metrics on manifolds.
\end{abstract}

\maketitle

\section{Introduction}

The aim of this paper is to study the reduction and integrability of stochastic dynamical systems (SDS) which are given by stochastic differential equations (SDE) on manifolds. Our motivation for doing this come in particular from the problem of physics-like second-order stochastic models of the prices of financial assets, which will be presented in a separate work \cite{Zung}. In this paper, we will write SDE's using the Stratonovich formulation: \begin{equation}\label{eq:Stratonovich}dx_t = X_0dt + \sum_{i=1}^m X_i \circ dB_{t}^i,\end{equation} where $x_t$ means a point on a manifold $M$ which depends on the time variable $t$, $X_0,X_1,\hdots,X_m$ are vector fields on $M$, and $B^1,\hdots,B^m$ denote independent Brownian motions (Wiener processes) on $\mathbb{R},$ see, e.g. \cite{Bismut1981,I1989,Ku1997}. Throughout this paper, we will use the notation \begin{equation}\label{notation}X = X_0 + \sum_{i=1}^m X_i \circ \dfrac{dB_{t}^i}{dt}\end{equation} to denote a stochastic dynamical system $X$  which is generated by the SDE \eqref{eq:Stratonovich}. The reason for using the Stratonovich formulation instead of the Ito calculus is that systems written in Stratonovich form behave well under transformations of coordinates, just like usual vector fields.

Recall (see e.g., \cite{Ok2003}) that the linear second-order differential operator
\begin{equation}
A_X=X_0 + \dfrac{1}{2}\sum X_i^2 
\end{equation}
is called the \textbf{generator}, or \textbf{diffusion operator}, of $X = X_0 + \sum_{i=1}^m X_i \circ \dfrac{dB_{t}^i}{dt}$. The meaning of this operator is that the stochastic process $x_t$ generated by $X$ on $M$ is completely determined by $A_X$. In particular, the formula
\begin{equation}
A_Xf(x) = \lim_{t \rightarrow 0}\dfrac{\mathbb{E}^x[f(x_t)]-f(x)}{t}, 
\end{equation}
holds for any function $f \in C^2(M,\mathbb{R})$ and any point $x\in M$, where $\mathbb{E}^x$ means the conditional expectation value under the initial condition $x_0 = x.$ 

If two SDS's $X = X_0 + \sum X_i \circ \dfrac{dB_{t}^i}{dt}$ and $Y = Y_0 + \sum Y_j \circ \dfrac{dB_{t}^j}{dt}$ have the same diffusion operator, i.e. $A_X = X_0 + \dfrac{1}{2}\sum X_i^2 = A_Y = Y_0 + \dfrac{1}{2}\sum Y_j^2$ then we will say that $X$ and $Y$ are \textbf{diffusion-equivalent}. In this paper, we consider SDS's mainly up to diffusion equivalence only.

Section 2 of this paper is devoted to the problem of reduction of SDS's. Recall that, for deterministic dynamical systems, there are basically two ways to reduce their dimension: 1) retriction to an invariant submanifold (say given by a level set of a first integral), and 2) projection to a quotient space (say with respect to a system-preserving action of a Lie group). Here we want to do the same thing, but for SDS's. 

The notions of first integrals and invariant submanifolds can be extended in a natural way from the category of deterministic systems to the category of stochastic systems, and they have been already studied by many people. We will recall them briefly in Subsection 2.1. The notion of symmetry groups for SDS's is more tricky. Many authors before us studied only strict symmetries, i.e. they require all the vector fields $X_0,X_1,\hdots,X_m$ in the expression (\ref{notation}) of an SDS $X$ to be invariant with respect to an action of a Lie group $G$, see e.g. \cite{AF1995,Gal,Or22007,Mi1999,PoRo}. If it is the case then of course the SDS $X$ can be projected to the quotient manifold $M\slash G$. However, from the point of view of applications, one can not really distinguish between two different SDS's $X = X_0  +  \sum X_i \circ \dfrac{dB^{i}_t}{dt}$ and $Y = Y_0  + \sum Y_j \circ \dfrac{dB^{j}_t}{dt}$ if their associated diffusion processes are the same, i.e. they have the same diffusion operator $A_X = X_0 + \dfrac{1}{2}\sum X_i^2  = A_Y = Y_0 + \dfrac{1}{2}\sum Y_j^2$. Moreover, there is a symmetry breaking phenomenon: while many natural stochastic processes (e.g. the Brownian motion on $\mathbb{R}^n$) have a large symmetry group, when one writes them as SDS's then the corresponding vector fields admit very little symmetry. That's why  we think that it is more natural to do reduction of SDS's diffusion-wise, i.e. only up to diffusion equivalence. The main result of Subsection 2.2 (Theorem \ref{thm:proj}) says that, if an SDS is diffusion-wise invariant with respect to a proper action of a Lie group $G$ on a manifold $M$ then it can be projected diffusion-wise to an SDS on $M\slash G$. A famous classical example is the Bessel process $Z = \dfrac{n-1}{2r}\partial_r + \partial_r \circ \dfrac{dB_t}{dt}$ on $\mathbb{R}_+$ ($n\in \mathbb{N}, r$ is the coordinate on $\mathbb{R}_+$), which is the diffusion-wise reduction of the Brownian motion of $\mathbb{R}^n$ with respect to the natural $SO(n)$-action.

In Subsection 2.3 we show that, unlike the deterministic case and the strict symmetry case, underlying geometric structures of an SDS are not conserved under diffusion-wise reduction in general. For example, it may happen that a reduced system of a Hamiltonian SDS does not admit any Hamiltonian structure. Subsection 2.4 contains a concrete example of diffusion-wise symmetry and reduction, namely the damped stochastic oscillator with respect to a natural $SO(2)$-action. Finally, in Subsection 2.5, we address the problem of  ``lost variables", i.e. the projectability of an SDS with respect to an arbitrary surjective map.

In Section 3 of this paper, we develope the notion of integrability of SDS's, taking hints from both the classical and quantum mechanics. Roughly speaking, an SDS on a manifold will be called integrable of type $(p,q,r)$ if there is a complete commuting family of $p$ diffusion operators, $q$ vector fields, and $r$ strong first integrals associated to it, with $p+q+r=\dim M$. Among other results, we show in this section an analogue of the classical Liouville's theorem, which says that under some mild assumptions, there is a natural torus action (called the \emph{Liouville torus action}) which preserves the system (Theorem \ref{thm:316}). Like in the classical case, this Liouville torus action is very important, because it has the fundamental structure-preserving property (see Theorem \ref{thm:action} and Theorem \ref{thm:actiondiff}), allows one to find action-angle coordinates and write down semi-local normal forms (see Theorem \ref{thm:0qr} and Theorem \ref{thm:1qr} for systems of types $(0,q,r)$ and $(1,q,r)$). Restricting integrable SDS's to common level sets of first integrals and reducing them with respect to Liouville torus actions, we get integrable systems of type $(p,0,0)$. In Subsection 3.5, we show how integrable systems of type $(p,0,0)$, i.e. given by $m$ commuting diffusion operators are related to compatible families of integrable Riemannian matrics on the manifold. Finally, in Subsection 3.6, the last subsection of this paper, we formulate a conjecture about reduced integrability of SDS's, which we do not know how to prove at this time.
\section{Stochastic dynamical systems  and their reduction}

\subsection{First integrals and invariant submanifolds}
Let us recall some definitions of first integrals and invariant submanifolds for stochastic dynamical systems (SDS), see, e.g.,\cite{Bismut1981,Or22007}.



\begin{defn}
Consider an SDS: \begin{equation}X = X_0 + \sum_{i=1}^k X_i \circ \dfrac{dB_{t}^i}{dt}\end{equation} on a manifold $M.$
 
i) A function $F: M \to \mathbb{R}$ is called a \textbf{strong first integral} of $X$ if $F$ is invariant with respect to $X_0, X_1,\hdots, X_k $, i.e. \begin{equation}X_0(F) = X_1(F) = \hdots = X_k(F) = 0.\end{equation}

ii) A function $F: M \to \mathbb{R}$ is called a \textbf{weak first integral} of $X$ if $F$ is invariant with respect to  the generator $\displaystyle A_X = X_0 + \dfrac{1}{2}\sum_{i=1}^k X_{i}^2$ of $X$, i.e. \begin{equation}A_X(F) = 0.\end{equation}

iii) A diffusion morphism $F: (M,X)  \to (\mathbb{R}, Z),$ where $Z$ is an SDS on $\mathbb{R}$, is called a \textbf{stochastic first integral} of $(M,X)$. In other words, $X$ can be projected to an one-dimensional system via $F$.

iv) A submanifold $N \subset M$ is called \textbf{invariant} with respect to $X$ if for any $x \in N$ we have $X_0(x), X_1(x), \hdots, X_k(x) \in T_x N,$ i.e all the vector fields $X_0, X_1,\hdots, X_k$ are tangent to $N.$ In other words, if $x \in N$ then almost surely we have $\phi_t(\omega,x) \in N$ for every $t$.
\end{defn}
We have the following characterization of strong first integrals:
\begin{prop}
A function $F: M \to \mathbb{R}$ is a strong first integral of an SDS $\displaystyle X = X_0  + \sum_{i=1}^k X_i \circ \dfrac{dB_{t}^i}{dt}$ if and only if \begin{equation}[A_X, F] = 0,\end{equation}
where in the above equation, $F$ is considered as a zeroth-order differential operator, i.e. multiplication by $F,$ and $A_X = X_0 + \dfrac{1}{2}\sum_{i=1}^k X_{i}^2$ is the diffusion operator of $X.$
\end{prop}

We observe that the notions of strong first integral, weak first integral, and invariant submanifold depend only on the diffusion equivalence class of SDS. More precisely, we have the following theorem, whose proof is straightforward:
\begin{thm}
Let $X = X_0 + \sum_{i=1}^k X_i \circ \dfrac{dB_{t}^i}{dt}$ and $Y = Y_0 + \sum_{i=1}^l Y_i \circ \dfrac{dW_{t}^i}{dt}$ be two SDS on a manifold $M$ which are diffusion equivalent. Then we have: 

i) A function $F: M \to \mathbb{R}$ is a strong first integral of $X$ if and only if it is a strong first integral of $Y$.

ii) A function $F: M \to \mathbb{R}$ is a weak first integral of $X$ if and only if it is a weak first integral of $Y.$

iii) A submanifold $N \subseteq M$ is invariant with respect to $X$ if and only if it is invariant with respect to $Y.$
\end{thm}
Let us mention here another obvious proposition, which is the same as in the case of deterministic systems:
\begin{prop}
Let $F_1,\hdots,F_q : M \to \mathbb{R}$ be strong first integrals of an SDS $\displaystyle X = X_0 + \sum_{i=1}^k X_i \circ \dfrac{dB_{t}^i}{dt}$ on a manifold $M$. Assume that the level set \begin{equation}N = \{F_1 =c_1,\hdots, F_q = c_q\}\end{equation} (where $c_1,\hdots,c_q$ are constants) is a submanifold of $M$. Then $N$ is an invariant submanifold of $X$ on $M.$
\end{prop}
\begin{remark}
Even if a deterministic system $X_0$ has a lot of first integrals, when stochastic terms $\displaystyle X_i \circ \dfrac{dB_{t}^i}{dt}$ are added to it, the resulting system $\displaystyle X = X_0 + \sum X_i \circ \dfrac{dB_{t}^i}{dt}$ will lose its first integrals. Especially, when the diffusion of the noise is nondegenerate, i.e. $A_X$  is elliptic, then there is no strong first integral at all. That's why one needs to weaken the notion of first integrals for SDS. Notice that a weak first integral is not necessarily a stochastic first integral, and vice versa a stochastic first integral is not necessarily a first integral. Given a stochastic first integral $F : (M,X) \to (\mathbb{R},Z)$, one can turn it into a weak first integral by composing it with a monotonous function $f: \mathbb{R} \to \mathbb{R}$ such that $f(Z)$ is a martingale process, so in a sense, stochastic integrals are stronger than weak first integrals. They are also more useful in reduction theory, because the morphism $F: (M,X) \to (\mathbb{R},Z)$ itself is a reduction to an 1-dimensional system. For example, in the works by Freidlin and other people, first integrals of a deterministic system become stochastic first integrals after a random perturbation and after taking an appropriate limit (also called variational equations), see, e.g., \cite{Fre1,Fre2}.
\end{remark}
\subsection{Reduction with respect to a symmetry group}

In this subsection, we study the reduction of an SDS on a manifold with respect to a Lie group action which preserves the system. First, let us recall the notion of morphisms between SDS's:
\begin{defn} Let $(M,X)$ and $(N,Y)$, where $X=X_0 + \sum_{i=1}^k X_i \circ \dfrac{dB_{t}^i}{dt}$ and $Y=Y_0 + \sum_{i=1}^m Y_i \circ \dfrac{dB_{t}^i}{dt}$ be two SDS on two manifolds M and N respectively.

i) A map $\Phi : (M, X) \to (N,Y)$ is called a \textbf{system morphism} if $m = k$ and $\Phi$ sends $X_i$ to $Y_i$ for all $i = 0,1,\hdots, k$, i.e for all $x \in M$ and $i = 0,1,\hdots, k$ we have \begin{equation}\Phi_{*}(X_i(x)) = Y_i{\Phi(x)}.\end{equation}
ii) A map $\Phi : (M, X) \to (N,Y)$ is called a  \textbf{diffusion morphism} if $\Phi$ sends the diffusion generator $A_X= X_0 + \dfrac{1}{2} \sum X^{2}_i$ of $X$ to the diffusion generator $A_Y = Y_0 + \dfrac{1}{2} \sum Y^{2}_i$ of $Y$, i.e for any function $f$ on $N$ we have \begin{equation}A_X(\Phi^*(f)) = \Phi^*(A_Y(f)).\end{equation}
\end{defn}

Of course, a system morphism is also a diffusion morphism, but the converse is not true in general. For example, if two different SDS on a manifold $M$ are diffusion equivalent, then the identity map is a diffusion morphism, but not a system morphism between them. Since system morphisms are in many cases too restrictive, we will often work with diffusion morphisms instead.

Let $\rho : G \curvearrowright M$ be an (effective) action of a Lie group $G$ on $M$. We will assume that either $G$ is compact, or $G$ is non-compact but the action is proper, so that the induced topology on the quotient space $M \slash G$ is Hausdorff. 

Recall that an SDS $X = X_0 + \sum X_i \circ \dfrac{dB^{i}_t}{dt} $ is called \textbf{invariant with respect to $G$} if all the vector fields $X_0, X_1,\hdots, X_k$ are invariant with respect to $G$. In other words, for every $g \in G$, the map $\rho_g: (M,X) \to (M,X)$ of the action $\rho$ of $G$ is a system isomorphism. In this case, it is well-known that $X= X_0 + \sum X_i \circ \dfrac{dB^{i}_t}{dt}$ can be projected to an SDS $Z = Z_0 + \sum Z_i \circ \dfrac{dB^{i}_t}{dt}$ on $M \slash G$ where $Z_i = \mbox{proj}_{M\slash G} (X_i)$ is the projection of $X_i$ on $M \slash G$ for each $i = 0,\hdots,k$. This is the starting point of the reduction theory of SDS's,  see, e.g.,\cite{AF1995,Gal,Or22007,Liao2008,Mi1999,PoRo}.

It may happen however that $G$ does not preserve the system $X = X_0 + \sum X_i \circ \dfrac{dB^{i}_t}{dt}$ but it preserves the diffusion process of the system, i.e. the diffusion generator $X_0 + \dfrac{1}{2} \sum X_{i}^2$ is invariant with respect to $G$. We still want to do reduction in this case.

\begin{defn}
We will say that an SDS $X = X_0  + \sum X_i \circ \dfrac{dB^{i}_t}{dt}$ on a manifold $M$ is \textbf{diffusion invariant} with respect to an action $\rho$ of a Lie group $G$ on $M$ if the operator $A_X = X_0 + \dfrac{1}{2} \sum X_{i}^2$ is invariant with respect to $G$. In other words,  for every $g \in G$, the map $\rho_g : (M,X) \to (M,X)$ of the action is a diffusion isomorphism.
\end{defn}

\begin{example} Put $X_0 = x \partial_y - y \partial_x, X_1 = \partial_x, X_2 = \partial_y$ on $\mathbb{R}^2$. Then the system $X = X_0 + X_1\circ \dfrac{dB_{t}^1}{dt} + X_2 \circ \dfrac{dB_{t}^2}{dt}$ is not invariant with respect to the rotation group $SO(2)$ but it is diffusion invariant with respect to $SO(2)$.\end{example}

A natural question  arises: given a system $X = X_0 + \sum X_i \circ \dfrac{dB^{i}_t}{dt}$ which is diffusion invariant with respect to $G$, does there exist a system $Y = Y_0 + \sum Y_j \circ \dfrac{dW^{j}_t}{dt}$ which is invariant with respect to $G$ and which is diffusion equivalent to $X$? 

In some cases, the answer is YES, but unfortunately, in many cases, the answer is NO, especially if the group $G$ is ``too big''. This phenomenon may be viewed as a {\bf symmetry breaking phenomenon}.

\begin{example}Let $M = S^{n-1} \subset \mathbb{R}^n$ be the unit $(n-1)$-dimendional sphere, $G = SO(n)$ which acts on $M$ by rotations. Then there is no non-trivial $SO(n)$-invariant vector field on $S^{n-1}$, but the Brownian motion on  $S^{n-1}$ (with respect to the usual metric) is of course $SO(n)$ invariant. In other words, the Brownian motion on  $S^{n-1}$  is $SO(n)$ invariant as a diffusion process, and it can be generated by an SDS, but there is no $SO(n)$ invariant SDS associated to it. \end{example}

Nevertheless, we can still do reduction, up to diffusion equivalence, of systems which are diffusion symmetric (i.e diffusion invariant with respect to a group action), as the following theorem shows:

\begin{thm} \label{thm:proj}
Let $X = X_0 + \sum_{i=1}^k X_i \circ \dfrac{dB_{t}^i}{dt}$ be an SDS on a manifold $M$ which is diffusion invariant with respect  to an action $\rho$ of  a compact Lie group $G$ on $M$. Then on the regular part of the quotient space $M \slash G$ there exists an SDS $Z = Z_0 + \sum_{i=1}^m Z_i \circ\dfrac{dW_{t}^i}{dt}$ for some $m \in \mathbb{N}$ whose generator $Z_0 + \dfrac{1}{2} \sum_{i=1}^m Z_{i}^2$ is the projection of the operator $X_0 + \dfrac{1}{2}\sum_{i=1}^k X_{i}^2,$ i.e. the projection map proj: $(M,X) \to (M\slash G, Z)$ is a diffusion morphism.
\end{thm}

\begin{proof}
First let us verify that $X_0 + \dfrac{1}{2}\sum_{i=1}^k X_{i}^2$ projects to a second-order differential operator on (the regular part of) $M \slash G$.

Let $f : M\slash G \to \mathbb{R}$ be a smooth function and denote by $\pi : M \to M\slash G$ the projection map. Then $\pi^* f$ is $G$-invariant. Since $A_X = X_0 + \dfrac{1}{2}\sum_{i=1}^k X_{i}^2$ is $G$-invariant, $A_X (\pi^* f)$ is also $G$-invariant, i.e. there is a unique function $ \hat{f}$ on $ M \slash G$ such that $A_X (\pi ^* f) = \pi ^* \hat{f}$. One then checks that, since $A_X$ is a second-order differential operator, the map $f \mapsto \hat{f}$ is also given by a second-order differential operator (i.e. in local coordinates, it depends on derivatives up to the second order only). 

Let us now prove the existence of $Z_0 + \sum_{i=1}^m Z_i \circ \dfrac{dW_{t}^i}{dt}$. First consider the case when there is  a global section $\mathcal{S} \subseteq M$ to the foliation by the orbits of $G$ in $M$. Then $M \slash G$ can be identified with $\mathcal{S}$. 

For each $i = 1,\hdots, k$ denote by $Z_i = \mbox{proj}_{\mathcal{S}}(X_i)$ the projection of the restriction of $X_i$ to $\mathcal{S}$ on $M \slash G$ along the orbits of $G$. More precisely, at each $y \in M \slash G$ denote by $y_{\mathcal{S}} \in \mathcal{S}$ the point in $\mathcal{S}$ such that $\Pi(y_{\mathcal{S}}) = y$ and put \begin{equation}Z_i(y) = \pi_{*}(X_i(y_{\mathcal{S}})).\end{equation}

Then one verifies easily that $\dfrac{1}{2}\sum_{i=1}^k Z_{i}^2$ has the same principal symbol as the projection $\mbox{proj}_{M\slash G}(X_0 + \dfrac{1}{2}\sum X_{i}^2)$ of $X_0 + \dfrac{1}{2}\sum X_{i}^2$ on $ M \slash G$. Hence the difference between these two differential operators is an order 1 differential operator (without zeroth order terms), i.e. a vector field, which we will denote by $Z_0$. Then $Z_0 + \sum Z_i \circ \dfrac{dB_{t}^i}{dt}$ will satisfy the conditions of the theorem.

In general, due to global topological obstructions, such a global section does not necessarily exist, but local sections exist, and we can use a partition of unity to construct our system on $M \slash G$ as follows:

Let $M\slash G = \bigcup_{k} U_k$ be a finite covering of $M$ by open (not necessarily connected) sets, together with a partition of unity $\sum f_k = 1$, where $f_k : M\slash G \to [0,1]$ is a function on $M\slash G$ whose support lies inside $U_k$, and such that over $U_k$ there is a section $\mathcal{S}_k$ to the $G$- foliation on $M$. Put \begin{equation}Z_{k,i} = \sqrt{f_k}\mbox{proj}_{\mathcal{S}_k}(X_i)\end{equation} (and extend it to the whole $M\slash G$ by putting it equal to $0$ outside $U_k$). Then $\dfrac{1}{2} \sum_{k,i} Z_{k,i}^2$ has the same principal symbol as $\mbox{proj}_{M\slash G}(X_0 + \sum X_{i}^2)$, and so there exists a vector field $Z_0$ on the regular part of $M \slash G$ such that $Z_0 + \sum_{k,i} Z_{k,i} \circ \dfrac{dW_{t}^{k,i}}{dt}$ satisfies the conditions of the theorem.
\end{proof}

\begin{defn} The system $Z_0 + \sum_{k,i} Z_{k,i} \circ \dfrac{dW_{t}^{k,i}}{dt}$ in Theorem \ref{thm:proj} will be called the \textbf{projection up to diffusion equivalence} of the system $X_0 + \sum X_i \circ \dfrac{dB_{t}^i}{dt}$ from $M$ to the quotient space $M\slash G $. \end{defn}

Of course, any other SDS on $M\slash G$ which is diffusion equivalent to $Z_0+ \sum Z_j \circ\dfrac{dW_{t}^j}{dt}$ will also be a projection of $X_0 + \sum X_i \circ\dfrac{dB_{t}^i}{dt}$  from $M$ to $M\slash G$ up to diffusion equivalence. Among all these diffusion equivalent systems, one may try to find a ``normal form'', i.e. a system whose expression is simplest possible.

\begin{remark} The space $M\slash G$ is an orbit space which is a stratified manifold with singularities in general. At the singularities of $M\slash G$, the projected SDS may also show a singular behavior, e.g. it may blow up.\end{remark}
\begin{example}

Consider the Brownian motion on $\mathbb{R}^n$ generated by $X = \sum_{i=1}^n \partial_{x_i} \circ \dfrac{dB_{t}^i}{dt}$. This system in diffusion-invariant with respect to the natural action of the group $SO(n)$. The quotient space is $\mathbb{R}^n\slash SO(n) \cong \mathbb{R}_{+}$ with the radial coordinate $r = \sqrt{\sum x_{i}^2}$. Simple calculations similar to the ones given above show that the reduced system on $\mathbb{R}_{+}$ is \begin{equation}\label{eq:Z}Z= \dfrac{n-1}{2r} \partial_r + \partial_r \circ \dfrac{dB_t}{dt}.\end{equation} In the literature, the diffusion process of the SDS $Z$ given by Formula \eqref{eq:Z} is called the \textbf{Bessel process}, and it is also defined as the process of $||W_n||$, where $W_n$ is an $n$-dimensional Brownian motion on $\mathbb{R}^n$, see, e.g., \cite{Ok2003}.
\end{example}
It is well-known that the Brownian motion (or more generally, any $SO(n)$-invariant Markov process) on $\mathbb{R}^n$ can be written as a semi-direct product of the Bessel process (or another process on $\mathbb{R}_{+})$ with a time-changed Brownian motion (or a more general time-changed process) on the sphere $\mathbb{S}^{n-1}$, see Galmarino \cite{Gal}. The semi-direct product result of Galmarino has been generalized by Pauwels and Rogers \cite{PoRo} and other people to a more general situation of a Markov process which is invariant with respect to an action of a Lie group $G$, under the hypothesis that there exists a global section to the foliation by the orbits of the action.

In our setting of an SDS X which is diffusion invariant with respect to a symmetry group $G$, this semi-direct product, also known as the decomposition of $X$ into the sum of an angular part (tangent to the orbits of the group action) with a radial part (transversal to the orbits) can be seen as follows:

\quad In the proof of Theorem \ref{thm:proj}, assume that there is a global section $\mathcal{S} \subseteq M$. Each vector field $Z_i$ on $\mathcal{S} \subseteq M, (i = 0,1,\hdots,k)$ can then be turned into a vector field on $M$ by left translation with respect to the action of $G$ and then by averaging over the isotopy group at each point. By doing so, we get a vector field $V_i$ on $M$ which is $G-$invariant and which projects to $Z_i$ on $\mathcal{S}$. One then verifies that the difference $L = (X_0 + \dfrac{1}{2}\sum X_{i}^2) - (V_0 + \dfrac{1}{2}\sum V_{l}^2)$ is tangent to the orbits of $G$, in the sense that for any function $f$ on $M$ and any $x \in M$, the value $L(f)(x)$ depends only on the restriction of $f$ to the orbits through $x$ of the action of $G$. Assuming that we can write $L = U_0 + \dfrac{1}{2}\sum_{j} U_{j}^2$, then $X = X_0 + \sum X_i \circ \dfrac{dB_{t}^i}{dt}$ is diffusion equivalent to the sum of $U_0 + \sum_{j} U_j \circ \dfrac{dB_{t}^j}{dt}$ and $V_0 + \sum_{l} V_l \circ \dfrac{dB_{t}^l}{dt}$. The parts $U_0 + \sum U_j \circ \dfrac{dB_{t}^j}{dt}$ and $V_0 + \sum V_l \circ \dfrac{dB_{t}^l}{dt}$ are called the \emph{angular part} and the \emph{radial part} of $X$, respectively.
\subsection{Structure-preserving SDS}
Besides having symmetry groups, deterministic systems may also preserve various geometric structures, e.g. volume forms, symplectic structures, Poisson structures, contact structures, etc. A natural question arises:  what are their stochastic analogs? What kind of SDS can also be said to preserve some geometric structure or to have some properties related to that structure?

For Hamiltonian systems, this question was studied by Bismut \cite{Bismut1981} and other people. One says that an SDS $\displaystyle X = X_0 + \sum_{i=1}^k X_i \circ \dfrac{dB_{t}^i}{dt}$ on a symplectic manifold $(M,\omega)$ is a Hamiltonian SDS if the vector fields $X_0, X_1,\hdots, X_k$ are Hamiltonian, i.e. there are $k+1$ functions $H_0, H_1,\hdots, H_k : M \to \mathbb{R}$ such that $X_i = X_{H_i}$ is the Hamiltonian vector field of $H_i$ for each $i$. Bismut \cite{Bismut1981} showed that the random flow of a Hamiltonian SDS preserves the symplectic structure almost surely. In fact, he proved the following more general theorem.

\begin{thm}[Bismut \cite{Bismut1981}]\label{thm:Bismut}
Let $\Lambda$ be an arbitrary smooth tensor field on a manifold $M$, i.e $\Lambda \in \Gamma (\bigotimes^p T^{*}M \bigotimes^q TM)$ for some $p,q \geq 0$, which is invariant with respect to the vector fields $X_0, X_1,\hdots, X_k$. Then the random flow of the SDS $X = X_0 + \sum_{i=1}^k X_i \circ \dfrac{dB_{t}^i}{dt}$ preserves $\Lambda$ almost surely.
\end{thm}

In particular, if $\Pi$ is a Poisson tensor, $f_0,\hdots,f_n : M \to \mathbb{R}$ are functions, and $X_i = df_i \lrcorner \Pi$ are their Hamiltonian vector fields with respect to $\Pi$, then the stochastic Hamiltonian system $X = X_0 + \sum_{i=1}^k X_i \circ \dfrac{dB_{t}^i}{dt}$ preserves the Poisson structure $\Pi$ according to Bismut's Theorem \ref{thm:Bismut}, because every vector field $X_i$ does. The symplectic leaves of the Poisson manifold $(M,\Pi)$ are invariant submanifolds of $X,$ and the Casimir functions of $(M,\Pi)$ are strong first integrals of $X$.
 
If an SDS $X$ on a manifold $M$ preserves a multi-vector field $\Lambda$, and both $X$ and $\Lambda$ are invariant with respect to a group action of a Lie group $G$ on $M$, then it is natural that the projection of $X$ to $M\slash G$ also preserves the projection of $\Lambda$ on $M\slash G$. In particular, in the case of  invariant stochastic Hamiltonian systems (where $\Lambda$ is the Poisson tensor), reduction theorems were obtained by Lazaro-Cami and Ortega \cite{Or22007}.


Notice that the structure-preserving property is not invariant under the diffusion equivalence: if $X$ and $Y$ are two diffusion equivalent SDS's, and $X$ preserves a tensor field $\Lambda$, it does not imply at all that the component vector fields of $Y$ also preserve $\Lambda$. Moreover, the reduction with respect to a diffusion-wise symmetry group may destroy the structure-preserving nature of a system. For example, it may happen that the original system is Hamiltonian, but the reduced system is not diffusion-equivalent to any Hamiltonian system.
\begin{example}
Consider the Hamiltonian SDS $\displaystyle X = X_h + \sum \partial_{x_i} \circ \dfrac{dB_{t}^i}{dt} + \sum \partial_{y_i} \circ \dfrac{dB_{t}^{i+n}}{dt}$
 on $(\mathbb{R}^{2n}, \omega = \sum dx_i \wedge dy_i),$ where $h = \dfrac{1}{2}\sum(x_{i}^2 + y_{i}^2)$, and $X_h$ generates a Hamiltonian $\mathbb{T}^{1}$-action on $\mathbb{R}^{2n}$ which preserves $X$ diffusion-wise. 
The reduced phase space is $\mathbb{R}^{2n}\slash \mathbb{T}^{1} \cong \mathbb{C}\mathbb{P}^{n-1} \times \mathbb{R}_{+}$ with symplectic leaves $\mathbb{C}\mathbb{P}^{n-1} \times \{pt\}$. Notice that the reduced symplectic forms on these sympletic leaves are not cohomologous to each other  and so a leaf can't be sent to a different leaf by a Poisson diffeomorphism, therefore any stochastic dynamical system on $\mathbb{R}^{2n}\slash \mathbb{T}^{1}$ which preserves the Poisson structure must also preserve each symplectic leaf. On the other hand, any such symplectic leaf will not be invariant with respect to the reduction of $X$ to $\mathbb{R}^{2n}\slash \mathbb{T}^1$, because its preimage under the projection map $\mathbb{R}^{2n} \to \mathbb{R}^{2n} \slash \mathbb{T}^1$ is a sphere $\mathbb{S}^{2n-1}$, 
which is of course not preserved by $X$ due to the noise term $\displaystyle \sum \partial_{x_i} \circ \dfrac{dB_{t}^i}{dt} + \sum \partial_{y_i} \circ \dfrac{dB_{t}^{i+n}}{dt}.$ 
\end{example}

\subsection{Example: the damped stochastic oscillator}\label{sec:damped}

The stochastic oscillator is such a ubiquitous model in science, that is has been studied extensively in the literature by many authors, see, e.g. \cite{Gitterman2005,Makus1988}.

As strange as it may sound, a damped oscillator (which loses its energy due to damping) may become conservative again (energy preserving) in a stochastic sense if a white noise is added to it. The reason is that the white noise has the effect of increasing the expected value of the energy, and so it cancels out the energy-losing effect of the damping term.

As an example of the reduction theory, in this subsection, we will describe a simple model of damped stochastic oscillator, and do the reduction of it with respect to a natural $SO(2)$ group action. We start with the harmonic oscillator given by the Hamiltonian function $h = \dfrac{1}{2}(x^2 + y^2) $ on the symplectic plane $(\mathbb{R}^2, \omega = dx \wedge dy)$. The Hamiltonian vector field is the rotation vector field $X_h = x\partial_y - y \partial_x$, which is of course $SO(2)$-invariant. We add a damping term, which is a vector field $D$ on $\mathbb{R}^2$ pointing towards the origin of $\mathbb{R}^2$ (so that $D(h) < 0$). For simplicity, we assume that $D$ is $SO(2)$-invariant, so we put $D = -f(r).(x \partial_x + y \partial_y)$ where $r = \sqrt{x^2+y^2}$ and $f$ is a positive function. Finally, we add a stochastic term to the system. The simplest one to take is the Brownian motion $B = \partial_x \circ \dfrac{dB_{t}^1}{dt} + \partial_y \circ \dfrac{dB_{t}^2}{dt}$. This term is also $SO(2)$-invariant diffusion-wise. Adding the above three terms, we get the system: \begin{equation*}X = X_h + D + B\end{equation*} \begin{equation}= (x \partial_y - y \partial_x) - f(r) (x\partial_x + y\partial_y) + (\partial_x \circ \dfrac{dB_{t}^1}{dt} + \partial_y \circ \dfrac{dB_{t}^2}{dt}).\end{equation}
Due to the damping term, $X$ is not a Hamiltonian SDS. On the other hand, it is $SO(2)$-invariant diffusion-wise, so we can reduce it to get a system on the orbit space $\mathbb{R}_{+} \cong \mathbb{R}^2\slash SO(2)$ with the coordinate $h = \dfrac{x^2 + y^2}{2}.$ Simple calculations show that the reduced system is: \begin{equation} Y = \Bigl (\dfrac{1}{2} - 2hf\Bigl )\partial_h + \sqrt{2h} \partial_h \circ \dfrac{dB_{t}}{dt}.\end{equation}

Using the coordinate $r = \sqrt{2h}$ instead of $h$, we obtain: \begin{equation} Y = \Bigl ( \dfrac{1}{2r} - rf\Bigl )\partial_r + \partial_r \circ \dfrac{dB_{t}}{dt} \end{equation}
Consider the deterministic part (the drift term) $Y_0 = g(r) \partial_r$ of $Y$, where \begin{equation} g(r) = \dfrac{1}{2r} - rf(r).\end{equation}
Assume that there is a value $r_0 > 0$ such that $g(r_0) = 0, g(r) < 0$ if $r > r_0$ and $g(r) > 0$ if $0<r< r_0$. (For example, if $f = c$ is a constant then $r_0 = \sqrt{\dfrac{1}{2c}}$, if $f= \dfrac{c}{r}$ so that the magnitude of the damping vector field $D$ is constant then $r_0 = \dfrac{1}{2c}$). Then $r_0$ is an attractive stationary point of the deterministic part $Y_0$ of $Y$. If the random term $\partial_r \circ \dfrac{dB}{dt}$ of $Y$ would dissapear, then the energy level of the system would tend to the value $\dfrac{r_{0}^2}{2}$ when time goes to infinity, and one would say that $\dfrac{r_{0}^2}{2}$ is the stable energy level of the system. With the random term $\partial_r \circ \dfrac{dB}{dt}$ in $Y$, $\dfrac{r_{0}^2}{2}$ is still a stable energy level of the system, but in a weaker sense: there is a probability density function $p$ on $\mathbb{R}_{+}$ which is ``mostly concentrated near $r_o$'' and which is invariant under the diffusion process of $Y$, similarly to the case of Ornstein-Uhlenbeck process. 

In polar coordinates $(\theta,r)$, i.e. $x = \cos\theta . r$ and $y = \sin\theta . r$, we have $D = rf(r) \partial_r, X_h = \partial_\theta$ and $(\partial_x)^2 + (\partial_y)^2 = \dfrac{1}{r}\partial_r + (\partial_r)^2 + \Bigl (\dfrac{1}{r}\partial_\theta\Bigl )^2,$ so our system $X = X_h + D +B$ is diffusion equivalent to \begin{equation} \hat{X} = \left[ (\dfrac{1}{2r} - rf) \partial_r + \partial_r \circ \dfrac{dB_{t}^1}{dt}\right] + \left[ \partial_\theta + \dfrac{1}{r} \partial_\theta \circ \dfrac{dB_{t}^2}{dt}\right]\end{equation}

The above expression is the decomposition of $\hat{X}$ (i.e of $X$ up to diffusion equivalence) into the sum of its radial part (which is nothing but the reduced system $Y$) and its angular part \begin{equation} \Theta = \partial_\theta + \dfrac{1}{r}\partial_\theta \circ \dfrac{dB_t}{dt}.\end{equation}
Since $\dfrac{1}{r}$ is invariant with respect to $\partial_\theta$, and $\partial_\theta$ commutes with $\dfrac{1}{r}\partial_\theta$, we immediately get the following fact about the movement of the angular coordinate in the model:
\begin{prop}
In the above model, consider $\theta$ as a coordinate function on $\mathbb{R}$ instead of on $\mathbb{T}^1 = \mathbb{R}\slash 2\pi\mathbb{Z}$ (i.e. consider the total angular movement). Then $\theta(t) - t$ is a martingale process. In particular, the mean frequency of the above damped stochastic oscillator model is $1$ almost surely for any initial value.
\end{prop}
\begin{example}\label{e:c}
Consider the case $f(r) = c$. In this case, we have \begin{equation}Y = \Bigl ( \dfrac{1}{2r}-cr\Bigl ) \partial_r + \partial_r \circ \dfrac{dB_t}{dt}.\end{equation}

By the Fokker-Planck equation, we get that the invariant density is $p(r) = 2cre^{-cr^2}$ on $\mathbb{R}^{+}$ with mean $\sqrt{\dfrac{\pi}{4c}}$, median $\sqrt{\dfrac{\ln 2}{c}}$, and mode $\sqrt{\dfrac{1}{2c}}$, (see Figure \ref{fig:g2}).
\begin{figure}[htb]
\begin{center}
\includegraphics[scale=0.9]{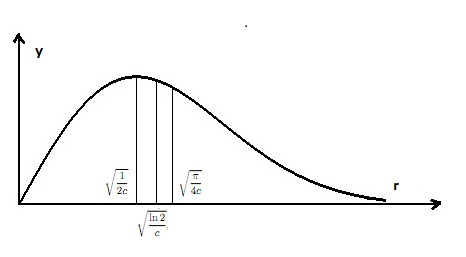}
\end{center}
\caption{Density function $p(r) = 2cre^{-cr^2}$for the case $f(r)=c$}
\label{fig:g2}
\end{figure}
\end{example}
\subsection{The problem of lost variables}\label{section:lost}

Consider an SDS $X$ on a manifold $M$, and a submersion $\Phi: M \to N$, with $\dim M = m > \dim N = k$. Locally, for each point $p \in M$, we can have a coordinate system $(x_1,\hdots,x_k,x_{k+1},\hdots,x_m)$ near $p$ in $M$, such that $x_1,\hdots,x_k$ are (the pull-back via $\Phi$ of) functions on $N$ and form a local coordinate system in a neighborhood of $\Phi(p)$ in $N$, and $x_{k+1},\hdots,x_m$ are additional coordinate functions on $M$.

Assume now that for each point $p \in M$, we have a way to determine $x_1(p),\hdots,x_k(p)$ with precision, but we do not know about $x_{k+1}(p),\hdots, x_m(p)$ and it is somehow impossible to measure them accurately. For example, if $x_1$ is the price of a stock, and $x_{k+1}$ is its momentum (assuming that such a thing exists), then $x_1$ can be observed (say via real-time streaming quotes), while $x_{k+1}$ is something for which there is no direct measurements, only some estimations based on various theories and formulas. This is what we call the problem of lost variables: how can we deal with a (stochastic or deterministic) dynamical system, when some of the variables are ``lost''?

A naive way to deal with this problem of lost variables is to simply forget about them: instead of considering the system as a system on $M$, we consider it as a system on $N$. One might argue that, by doing so, we will get a system on $N$, which is ``more random'' than the system on $M$ (if the system on $M$ is deterministic, the system on $N$ would still be random) due to lack of information. In other words, the submersion map $\Phi: M \to N$ would send our system $X$ on $M$ to a random dynamical system on $N$, i.e. we would get a reduction by simply forgetting (losing) some variables.

Unfortunately, things do not work that way in general. In other words, if $\Phi: M \to N$ is given and if $X$ is an arbitrary SDS on $M$, then in general there does not exist any random dynamical system $Z$ on $N$ such that $\Phi : (M,X) \to (N,Z)$ would be a diffusion morphism, i.e. a morphism between the two corresponding stochastic processes. One cannot do reduction of SDS by simply forgetting some (arbitrary) variables!

As a matter of fact, we have the following straightforward and rather restrictive condition for an SDS $X$ on $M$ to be projectable to a system on $N$ via a given surjective map $\Phi : M \to N$:
\begin{prop}\label{pro:proj}
Let $\Phi : M \to N$ be a smooth surjective map from a smooth manifold $M$ to a smooth manifold $N$, and $\displaystyle X = X_0 + \sum_{i=1}^k X_i \circ \dfrac{dB_{t}^i}{dt}$ be an SDS on $M$.

i) There exists an SDS $\displaystyle Y = Y_0 + \sum_{i} Y_i \circ \dfrac{dB_{t}^i}{dt}$ on $N$ such that $\Phi : (M,X) \to (N,Y)$ is a system morphism if and only if for any points $x,y \in M$ such that $\Phi(x) = \Phi(y)$ we also have \begin{equation}\label{eq:condition1} \Phi_{*}(X_i(x)) = \Phi_{*}(X_i(y)) \quad \forall i = 0,1,\hdots,k.\end{equation}
ii) The diffusion process of $X$ is projectable to a Markov process on $N$ if and only if for any function $f : N \to \mathbb{R}$ and any two points $x,y \in M$ such that $\Phi(x) = \Phi(y)$ we also have \begin{equation}\label{eq:condition2} A_X(\Phi^*(f)(x)) = A_X(\Phi^*(f)(y))\end{equation} where $A_X$ is the diffusion generator of $X$. If this condition is satisfied and $\Phi$ is a submersion then the projected diffusion process on $N$ is a diffusion process generated by an SDS on $N$.

iii) In the case when $X= X_0$ is a smooth deterministic system then the deterministic process generated by $X$ on $M$ is projectable to a Markov process on $N$ if and only if for any points $x,y \in M$ such that $\Phi(x) = \Phi(y)$ we also have \begin{equation} \label{eq:condition3} \Phi_{*}(X(x)) = \Phi_{*}(X(y)). \end{equation} If this condition is satisfied then $X$ is projected to a smooth vector field on $N$.
\end{prop}
\begin{example}
Put $M = \mathbb{T}^2$ with two periodic coordinates $\theta_1$(mod 1), $\theta_2$(mod 1), $N = \mathbb{T}^1$ with the periodic coordinate $\theta_2$ (mod 1), $\Phi : M \to N$ given by the formula $\Phi(\theta_1,\theta_2) = \theta_2$. The vector field $X = \sin(2\pi\theta_1)\partial_{\theta_2}$ on $M$ does not satisfy the condition of Assertion iii) of above proposition, so the deterministic process $X$ on $M$ cannot be projected to any Markov process on $N$.
\end{example}

\section{Integrable stochastic dynamical systems}
\subsection{Integrable dynamical systems and Liouville torus actions}
In this subsection, we will recall the definition of integrability of deterministic dynamical system, and show the fundamental structure-preserving property of their corresponding Liouville torus actions. In particular, Theorem \ref{thm:actiondiff} of this subsection about the invariance of linear differential operators under Liouville torus actions will be needed in the study of integrable SDS. 
\begin{defn}
A vector field $X$ on a manifold $M$ is said to be \textbf{integrable of type (p,q)}, where $p \geq 1,  q \geq 0, p+q = \dim M$, if there exist $p$ vector fields $X_1, X_2,\hdots,X_p$ and $q$ functions $F_1,\hdots,F_q$ on $M$ which satisfy the following conditions:

i) The vector fields $X_1,\hdots,X_p$ commute pairwise and commute with $X$: \begin{equation}[X_i, X_j] = 0 \text{ and } [X, X_i] = 0 \quad \forall i, j.\end{equation}

ii) The functions $F_1,\hdots, F_q$ are common first integrals of $X, X_1,\hdots, X_p$: \begin{equation}X(F_j) = X_i(F_j) = 0,\quad  \forall i, j.\end{equation}

iii) $X_1 \wedge X_2 \wedge \hdots \wedge X_p \ne 0$ and $dF_1 \wedge \hdots \wedge dF_q \ne 0$ almost everywhere.

Under the above conditions, we will also say that $(X_1,\hdots,X_p, F_1,\hdots,F_q)$ is an integrable system of type $(p,q)$.
\end{defn}

A level set  \begin{equation} N = \{ F_1 = c_1,\hdots, F_q = c_q\}\end{equation} of an integrable system $(X_1,\hdots,X_p,F_1,\hdots,F_q)$ is called a \textbf{regular level set} if $X_1 \wedge \hdots \wedge X_p(x) \ne 0$ and $dF_1 \wedge \hdots \wedge dF_q (x) \ne 0$ $\forall$ $x \in N$. We have the following theorem, which goes back to Liouville \cite{Liu1855}. (Liouville proved it for the Hamiltonian systems, but the proof in the non-Hamiltonian case is essentially the same):

\begin{thm}[Liouville]\label{thm:Liou}
Let $N$ be a connected compact regular level set of an integrable system $(X_1,\hdots, X_p, F_1,\hdots, F_q)$. Then $N$ is diffeomorphic to a $p$-dimensional torus $\mathbb{T}^p$, and a neighborhood  $\mathcal{U}(N)$ of $N$ can be written as $\mathbb{T}^p \times B^q$ (where $B^q$ is  a $q$-dimensional ball) with a coordinate system $(\theta_1,\hdots, \theta_p, r_1,\hdots, r_q)$ where $\theta_1,\hdots, \theta_p$ are periodic coordinates on $\mathbb{T}^p,$ such that: \begin{equation}X_i = \sum a_{ij}(r_1,\hdots,r_q) \dfrac{\partial}{\partial \theta_j}\end{equation} and \begin{equation}F_i = f_i(r_1,\hdots, r_q)\end{equation} do not depend on $\theta_1,\hdots, \theta_p.$
In particular, the transitive $\mathbb{T}^p$-torus action $$\mathbb{T}^p \times (\mathbb{T}^p \times B^q) \to \mathbb{T}^p \times B^q$$ \begin{equation}((\rho_1,\hdots, \rho_p), (\theta_1,\hdots,\theta_p, r_1,\hdots,r_q)) \mapsto (\rho_1+\theta_1,\hdots, \rho_p+\theta_p, r_1,\hdots, r_q)\end{equation} preserves the system.
\end{thm}
The $\mathbb{T}^p$ action in the above theorem is unique up to automorphisms of $\mathbb{T}^p$ and is called the \textbf{Liouville torus action} of the system near the \textbf{Liouville torus} $N$.

A very important fact about the Liouville torus action is that it preserves every tensor field which is preserved by the system. More precisely, we have:
\begin{thm}[\cite{Zung2012}] \label{thm:action}
 Under the assumptions of Theorem \ref{thm:Liou}, let $\mathcal{G}$ be an arbitrary tensor field on $M$, $\mathcal{G} \in \Gamma(\bigotimes^k TM \bigotimes^h T^{*}M)$, which satisfies at least one of the following two conditions:

i) $\mathcal{G}$ is invariant with respect to the vector field $X_1,\hdots, X_p$: \begin{equation}\mathcal{L}_{X_i} \mathcal{G} = 0 \quad \forall i = 1,\hdots,p.\end{equation}

ii) $\mathcal{G}$ is invariant with respect to the vector field $X_1$, and moreover, the orbit of $X_1$ is dense in almost every orbit of the Liouville $\mathbb{T}^p$-action near $N$. 

Then $\mathcal{G}$ is also invariant with respect to the Liouville $\mathbb{T}^p$-action in a neighborhood of $N$.
\end{thm}

For example, in the case of an integrable Hamiltonian system, the symplectic form $\omega$ is a covariant tensor which is preserved by the system, so the above theorem says that $\omega$ is also preserved by the Liouville $\mathbb{T}^p$-torus action. One recovers easily  the existence of action-angle variables (the so called Arnold-Liouville-Mineur theorem) from this fact.

Since every SDS gives rise to a second order differential operator (the diffusion operator), a natural step in generalizing Theorem \ref{thm:action} to the case of integrable SDS is to replace the invariant tensor field $\mathcal{G}$ in the theorem by an invariant differential operator. By doing so, we get the following theorem: 

\begin{thm}\label{thm:actiondiff}
Under the assumptions of Theorem \ref{thm:Liou}, let $\Lambda$ be a linear differential operator on $M$ which satisfies at least one of the following two conditions :

i) $\Lambda$ is invariant with respect to $X_1,\hdots, X_p.$

ii) $\Lambda$ in invariant with respect to $X_1$, and moreover, the orbit of $X_1$ is dense in a dense family of orbits of the Liouville $\mathbb{T}^p$-action near $N.$

Then $\Lambda$ is invariant with respect to the Liouville $\mathbb{T}^p$-action in a neighborhood of $N.$
\end{thm}

Of course, Theorem \ref{thm:action} and Theorem \ref{thm:actiondiff} look very similar to each other, except for the fact that Theorem \ref{thm:action} deals with tensor fields while Theorem \ref{thm:actiondiff} deals with linear differential operators. Recall that differential operators of order one are given by vector fields. But higher-order differential operators are not given by tensor fields, so Theorem \ref{thm:actiondiff} is not a consequence of Theorem \ref{thm:action}. To prove Theorem \ref{thm:actiondiff}, we will prove the following:

\begin{thm}\label{thm:dense}
Let $Z = \sum_{i=1}^p a_i(r_1,\hdots, r_q)\partial_{\theta_i}$ be a vector field on $\mathcal{U} = \mathbb{T}^p \times B^q$ with coordinates $(\theta_1$ (mod 1),$\hdots$, $\theta_p$ (mod 1), $r_1,\hdots,r_q)$ which is tangent to the tori $\mathbb{T}^p \times \{\mbox{pt}\}$ and invariant with respect to the natural free action of $\mathbb{T}^p$ on $\mathcal{U}=\mathbb{T}^p \times \mathbb{B}^q$. Assume that the functions $a_1,\hdots, a_p$ are incommensurable, i.e there does not exist any nontrivial $p$-tuple of integers $(k_1,\hdots, k_p)$ such that $\sum k_i a_i(r) = 0$ in an open subset of $\mathcal{U}$. Let $\Lambda$ be a linear differential operator on $\mathcal{U}$ which is invariant with respect to $Z$, i.e \begin{equation}\mathcal{L}_Z \Lambda = Z \circ \Lambda - \Lambda \circ Z = 0.\end{equation} Then $\Lambda$ is invariant with respect to the action of $\mathbb{T}^p$ on $\mathcal{U}.$
\end{thm}
\begin{proof}
The proof follows the same steps as Zung's proof in \cite{Zung2012} of Theorem \ref{thm:action}. We write $\Lambda$ in the coordinates $(\theta,r)$ as follows: \begin{equation}\Lambda = \sum_{I,J} C_{I,J}(\theta,r) \partial_{\theta}^I \partial_{r}^J,\end{equation} where $I = (i_1,\hdots, i_p)$ and $J = (j_1,\hdots,j_q)$ are multi-indexes and \begin{equation}\partial_{\theta}^I=(\partial_{\theta_1})^{i_1} (\partial_{\theta_2})^{i_2} \hdots (\partial_{\theta_p})^{i_p}.\end{equation} The principal symbol of $\Lambda$ is the function \begin{equation}\sum_{|I| + |J| = m} C_{I,J} \hat{\theta}^{I} \hat{r}^J,\end{equation} where $|I| = \sum_{k=1}^p i_k$ , $ m$ is the order of $\Lambda$, $\hat{\theta}^I = \hat{\theta}^{i_1}_1 ... \hat{\theta}^{i_p}_p$ and $\hat{\theta}_i : T^* M \to \mathbb{R}$ are linear functions on $T^* M$ given by $\partial_{\theta_i}.$

Let us first prove that the principal symbol of $\Lambda$ is invariant with respect to the $\mathbb{T}^p$-action.

Denote by $\mathcal{O}_m$ the space of linear differential operators of order at most $m$ on $\mathcal{U} = \mathbb{T}^p \times B^q.$ Then the principal symbol of $\Lambda$ may be identified with an element of the quotient $\mathcal{O}_m \slash{\mathcal{O}_{m-1}}$. We will make some calculations modulo $\mathcal{O}_{m-1}:$
$$\mathcal{L}_Z \Lambda \equiv \left(\sum_i a_i \partial_{\theta_i}\right) \left(\sum_{|I|+|J| = m}C_{I,J} \partial_{\theta}^I \partial_{r}^J\right) - \left(\sum_{|I|+|J| = m}C_{I,J} \partial_{\theta}^I \partial_{r}^J\right)  \left(\sum_i a_i \partial_{\theta_i}\right)$$
$$\equiv \sum_{i,I,J}a_i \dfrac{\partial C_{I,J}}{\partial \theta_i} \partial_{\theta}^I \partial_{r}^J - \sum_{|I|+|J| = m} \sum_{i=1}^p \sum_{k=1}^q C_{I,J} \dfrac{\partial a_i}{\partial r_k}\partial_{\theta}^{I + 1_i} \partial_{r}^{J-1_k}\quad \mbox{ mod }\mathcal{O}_{m-1},$$ where $1_i = (0,\hdots,1,\hdots,0)$ with $1$ at the $i$-th place.

Consider a multi-index $(I,J)$ such that $|J| = h$ is the highest possible. Then the coefficient of $\partial_{\theta}^I \partial_{r}^J $ in the above expression of $\mathcal{L}_X \Lambda$ is \begin{equation}\sum_{i} a_i \dfrac{C_{I,J}}{\partial \theta_i} = Z(C_{I,J}).\end{equation}
Since $\mathcal{L}_X \Lambda = 0$, all of its coefficients must vanish, and in particular $Z(C_{I,J}) = 0$, i.e. $C_{I,J}$ is invariant with respect to $Z$. In other words, it is invariant on the orbits of $Z$. But the incommensurability condition on $Z$ implies that its orbits are dense on a dense family of tori $\mathbb{T}^p \times \{\mbox{pt}\}$, i.e. it is invariant under the $\mathbb{T}^p$- action.

Consider now a multi-index $(I,J)$ such that $| J | = h-1$ is lower than the highest possible number $h$ by $1$. Then the coefficient of $\partial_{\theta}^I \partial_{r}^J$ in the above expression of $\mathcal{L}_X \Lambda$ is: 
\begin{equation}\sum_{i} a_i \dfrac{\partial C_{I,J}}{\partial \theta_i} - \sum_{j,k} C_{I - 1_j, J + 1_k} \dfrac{\partial a_j}{\partial r_k},\end{equation} which must be zero, so we get the following equality : 
\begin{equation}Z(C_{I,J}) = \sum_{j,k} C_{I - 1_j, J + 1_k} \dfrac{\partial a_j}{\partial r}.\end{equation} Notice that the right hand side of this equation is $\mathbb{T}^p$- invariant, because the functions $a_j$ are $\mathbb{T}^p$ invariant and the functions $C_{I - 1_j, J + 1_k}$ are also $\mathbb{T}^p$- invariant by the previous argument. Thus $Z(C_{I,J})$ is $\mathbb{T}^p$- invariant. But the mean value of $Z(C_{I,J})$ on each torus $\mathbb{T}^p$ is $0$, so in fact we have $Z(C_{I,J}) = 0$, which, as we have seen, implies that $C_{I,J}$ is $\mathbb{T}^p$- invariant.

Similarly,  by induction on $h - |J|$, we get that $C_{I,J}$ is $\mathbb{T}^p$-invariant for any multi-index $(I,J)$ such that $|I| + |J| = m$ is the order of $\Lambda$.

Denote by $\mathcal{O}_{\mathbb{T}^p}$ the space of $\mathbb{T}^p$-invariant operators and consider $\Lambda$ modulo $\mathcal{O}_{\mathbb{T}^p}$. Then since all the coefficients of order $m$ of $\Lambda$ are $\mathbb{T}^p$- invariant, we have : 
\begin{equation}\Lambda \equiv  \sum_{|I| + |J| \leq m-1} C_{I,J} \partial_{\theta}^I \partial_{r}^J \quad \mbox{ mod } \mathcal{O}_{\mathbb{T}^p}.\end{equation}

Repeat the above process for the terms of order $m-1$ of $\Lambda$, we get that $Z(C_{I,J})$ is $\mathbb{T}^p$- invariant for any such multi-index $(I,J)$. By reverse induction on $|I| + |J|$, using the same arguments as above, we get that $C_{I,J}$ is $\mathbb{T}^p$-invariant for any multi-index $(I,J)$. Thus $\Lambda$ is $\mathbb{T}^p$-invariant.
\end{proof}
\textbf{Proof of Theorem \ref{thm:actiondiff}} : If the orbit of the vector field $X_1$ is dense in a dense family of orbits of the Liouville $\mathbb{T}^p$-action, then $X_1$ satisfies the condition of $Z$ in Theorem \ref{thm:dense}. Thus, Part (ii) of Theorem \ref{thm:actiondiff} follows immediately from Theorem \ref{thm:dense}. 

The proof of Part (i) is similar. Part (i) also follows easily from Part (ii), because we can choose constant coefficients $a_i$ such that the orbit of the vector field $Y = \sum a_i X_i$ is dense in a dense family of orbits of the Liouville $\mathbb{T}^p$-action.{\hfill$\square$}

\subsection{What is an integrable SDS?}
Up to diffusion equivalence, an SDS $\displaystyle X = X_0 + \sum X_i \circ \dfrac{dB_{t}^i}{dt}$ is characterized by its diffusion operator $A_X = X_0 + \dfrac{1}{2}\sum_{i=1}^k X_{i}^2$, which is a second order linear differential operator. So an SDS may be viewed as something lying between classical dynamical systems (first order differential operators) and quantum systems which are often given by (pseudo-)differential operators. For quantum systems, the notion of integrability usually means the existence of a full family of commuting operators. Taking hints from both classical and quantum mechanics, we arrive at the following notion of integrability for SDS's:
\begin{defn}\label{defn:pqr}
\textbf{An integrable SDS of type (p,q,r)} on a manifold $M$, where $p \geq 1, q \geq 0, r \geq 0, p+q+r = \dim M$, consists of a family of:

\quad$\bullet$ p diffusion generators: $\Lambda_1,\hdots,\Lambda_p$ (which can be written as $\Lambda_i = X_i^0 + \dfrac{1}{2}\sum_{k=1}^{h_i} (X_{k}^i)^2,$ $ i= 1,\hdots, p, \mbox{ where } X_{0}^i,\hdots, X_{h_i}^i \mbox{ are vector fields),}$

\quad$\bullet$ q vector fields: $Z_1,\hdots, Z_q,$

\quad$\bullet$ r functions: $F_1,\hdots, F_r$\\
on $M$,  which commute pairwise when considered as linear differential operators on $M$ (the $\Lambda_i$ are second-order, the $Z_j$ are first order, and the $F_k$ are zeroth order operators respectively), and which satisfy the following independence condition : the principal symbols of these operators form a family of $p + q +r = \dim M$  functionally independent functions on $T^* M$.
\end{defn}

\begin{defn}\label{defn:pqr2}
An SDS $X$ on manifold $M$ is called \textbf{integrable} of type $(p,q,r)$ if there is an integrable system $(\Lambda_1,\hdots,\Lambda_p,Z_1,\hdots,Z_q,F_1,\hdots,F_r)$ of type $(p,q,r)$ for some $p,q,r$ $(p+q+r = \dim M)$ such that the diffusion generator $A_X$ of $X$ commutes with all the linear differential operators $\Lambda_1,\hdots,\Lambda_p,Z_1,\hdots,Z_q,F_1,\hdots,F_r$. We will also say that $X$ is integrable with the aid of an integrable SDS $(\Lambda_1,\hdots,\Lambda_p,Z_1,\hdots,Z_q,$ $F_1,\hdots,F_r)$.
\end{defn}
Of course, in the above definition, one can often put $\Lambda_1 = A_X$. But it may also happen that $p=0$. If, in the above definition, the $\Lambda_i$ are diffusion operators linear differential operators of any order instead of diffusion operators, then we will say that we have a SDS which is \textbf{integrable in quantum sense}.

Recall that, if $Z$ is a vector field, considered as a differential operator of order $1$ on $M$, then the principal symbol of $Z$ is $Z$ itself, but considered as a fiberwise linear function $\hat{Z} : T^* M \to \mathbb{R}$. If $F$ is a function on $M$, then it can also be considered as a zeroth-order operator on $M$ (multiplication with $F$), whose symbol is nothing but the pull-back $\hat{F}$ of $F$ from $M$ to $T^* M$. If $\Lambda_i = X_0 + \dfrac{1}{2} \sum X_{k}^2$ is a diffusion generator, then its principal symbol is $\hat{\Lambda} =  \dfrac{1}{2} \sum \hat{X}_{k}^2$ where each $\hat{X}_k  : T^* M \to \mathbb{R}$ is a fiberwise-linear function on $T^* M$ given by $X_k$. The independency condition in the above definition means that the functions $\hat{\Lambda}_1,\hdots, \hat{\Lambda}_p, \hat{Z}_1,\hdots, \hat{Z}_q, \hat{F}_1,\hdots, \hat{F}_r $ are functionally independent on $T^* M.$

Let us recall the following classical fact from the theory of linear differential operators, see, e.g.,\cite{LiDiffOpe}: if two linear differential operators $\Lambda$ and $\Pi$ on a manifold $M$ commute then their principal symbols $\hat{\Lambda}$ and $\hat{\Pi}$ Poisson-commute on $T^*M$ with respect to the canonical symplectic structure. As an immediate consequence of this fact and Definition \ref{defn:pqr} we have the following relation between integrable SDS's and integrable Hamiltonian systems:
\begin{prop}\label{thm:(pqr)}
Let $(\Lambda_1,\hdots, \Lambda_p, Z_1,\hdots, Z_q, F_1,\hdots, F_r)$ be an integrable SDS of type $(p,q,r)$ on a manifold $M$. Then $(\hat{\Lambda}_1,\hdots, \hat{\Lambda}_p, \hat{Z}_1,\hdots, \hat{Z_q}, \hat{F}_1,\hdots, \hat{F}_r)$ is an integrable Hamiltonian system on $T^* M.$
\end{prop}
\begin{remark}
In the special case when $p = 0$ then an integrable SDS of type $(0,q,r)$ is in fact an integrable deterministic dynamical system of type $(q,r)$, and Proposition \ref{thm:(pqr)} is of course still valid in this case.
\end{remark}
\begin{example}
The damped stochastic oscillator in Subsection \ref{sec:damped} is an integrable SDS of type $(1,1,0).$ Integrable stochastic dynamical systems considered by Xue-Mei Li \cite{XuLi2008} in the theory of averaging of stochastic perturbations are integrable SDS's of type $(0,n,n)$ on a symplectic $2n$-dimensional manifold.
\end{example}

\begin{prop}\label{pro:pqr}
If a non-trivial SDS $X = X_0 + \sum_{i=1}^k X_i \circ \dfrac{dB_{t}^i}{dt}$ is integrable of type $(p,q,r)$ for some $p,q,r \geq 0$ then it is also integrable of type $(p+q+r,0,0).$
\end{prop}
\begin{proof}
Let $(\Lambda_1,\hdots,\Lambda_p,Z_1,\hdots,Z_q,F_1,\hdots,F_r)$ be an integrable SDS of type $(p,q,r)$ whose components commute with $X$. We want to construct a SDS $$(\Lambda_1,\hdots,\Lambda_p,\Lambda_{p+1},\hdots,\Lambda_{p+q},\Lambda_{p+q+1},\hdots,\Lambda_{p+q+r})$$ of type $(p+q+r,0,0)$ whose components commute with $X$. It can be done, for example, as follows: Put $\Lambda_{p+i} = Z_{i}^2 \quad(i=1,\hdots,q)$ and $\Lambda_{p+q+i} = F_i^2 \Lambda_1\quad (i=1,\hdots,r)$ if $p\geq 1$. (It is easy to see that if $\Lambda$ is a diffusion operator and $F$ is a real function then $F^2\Lambda$ is again a diffusion operator). If $p=0$ but $q\geq 1$ then we can put $\Lambda_{p+q+i} = (F_iZ_1)^2$ for example. The case $p=q=0$ is excluded because in that case $X$ would be trivial. The verification of functional independence of $\hat{\Lambda}_1,\hdots,\hat{\Lambda}_{p+q+r}$ is straightforward.
\end{proof}
\subsection{Existence of Liouville torus actions for  integrable SDS's}
\begin{defn}
Let $(\Lambda_1,\hdots,\Lambda_p,Z_1,\hdots,Z_q,F_1,\hdots,F_r)$ be an integrable SDS on a manifold $M$.

i) A point $x \in M$ will be called a \textbf{semi-regular point} of the system if $dF_1(x) \wedge \hdots \wedge dF_r(x) \ne 0$, and $\mbox{span} \hat{A}_1(x),\hdots,\mbox{span} \hat{A}_p(x), Z_1(x),\hdots,Z_q(x)$ together span the kernel space \begin{equation} \bigcap_{i=1}^p\textnormal{Ker } dF_i(x) = \{\alpha \in T_xM | \langle \alpha,dF_1(x)\rangle = \hdots = \langle \alpha,dF_r(x)\rangle = 0 \}.\end{equation} Here span $\hat{\Lambda}_i(x) \stackrel{def}{=} \textnormal{Vect } (Y_1(x),\hdots,Y_k(x))$ if $\hat{\Lambda} = \dfrac{1}{2}\sum_{i=1}^k \hat{Y}_{i}^2$ where $Y_i$ are vector fields and $\hat{Y}_i: T^*M \to \mathbb{R}$ are their corresponding symbols.

ii) If $x$ is a semi-regular point such that $Z_1(x)\wedge\hdots\wedge Z_q(x) \ne 0$ then $x$ is called a \textbf{regular point}.

iii) A connected level set $N = \{F_1 = c_1,\hdots,F_r = c_r\}$ is called a \textbf{regular level set} of the system if every point $x \in N$ is semi-regular, and almost every point of $N$ is regular.
\end{defn}
\begin{thm}\label{thm:316}
Let $(\Lambda_1,\hdots,\Lambda_p,Z_1,\hdots,Z_q,F_1,\hdots,F_r)$ be an integrable SDS on a manifold $M$ such that the map $(F_1,\hdots,F_r) : M \to  \mathbb{R}^r$ is proper. Then for any connected regular level set $N = \{F_1=c_1,\hdots,F_r=c_r\}$ of the system, there is a torus $\mathbb{T}^l$-action $\rho : \mathbb{T}^l \times \mathcal{U}(N) \to \mathcal{U}(N)$ in a neighborhood $\mathcal{U}(N)$ of $N$, where $l \geq q$, which preserves the system, and such that the orbits of this $\mathbb{T}^l$-action on $N$ are saturated by the orbits of the $\mathbb{R}^q$-action generated by $Z_1,\hdots Z_q$ on $N$. (They do not necessarily coincide outside of $N$).
\end{thm}
\begin{remark}
In analogy with the case of deterministic integrable systems, we will call the torus $\mathbb{T}^l$-action in $\mathcal{U}(N)$ provided by the above theorem the \textbf{Liouville torus action}.
\end{remark}

\begin{proof}
Due to the regularity and properness condition, $N$ is a compact submanifold of $M$ of dimension $p+q$, and moreover we have that the second-order operator
\begin{equation}
\Delta_N = \sum_{i=1}^p \Lambda_i + \sum_{j=1}^q Z_{j}^2
\end{equation}
on $N$ (i.e. restricted to the functions on $N$) is an elliptic operator on $N$. Hence there is a unique Riemannian metric $g_0$ on $N$ and a vector field $V_N$ on $N$ such that \begin{equation} \Delta_N = V_N + \Delta_{g_0},\end{equation} 
where $\Delta_{g_0}$ denotes the Laplace-Beltrami operator of $g_0$. Since $Z_1,\hdots,Z_q$ preserve $\Delta_N$, they must also preserve the principal symbol, i.e. they preserve the Riemannian metric $g_0$. Since $N$ is compact, the group $\mathcal{O}(g_0)$ of isometries of $g_0$ is a compact Lie group. The Abelian subgroup $\displaystyle \exp \Bigl ( \sum_{i=1}^q t_i Z_i|_{N}\quad| t_i \in \mathbb{R}\Bigl )$ of diffeomorphisms of $N$ is a connected Abelian subgroup of $\mathcal{O}(g_N)$ of dimension $q$ (because the vector fields $Z_1,\hdots,Z_q$ are independent almost everywhere on $N$), hence its closure 
\begin{equation}T_0 := \overline{\exp \Bigl ( \sum_{i=1}^q t_i Z_i|_{N}\quad| t_i \in \mathbb{R}\Bigl )}
\end{equation} 
is a torus of dimension $l \geq q$.

Denote by $\mathcal{U}(N)$ a sufficiently small tubular neighborhood of $N$ in $M$ which is saturated by connected level sets of $(F_1,\hdots,F_r)$ and such that every point in $\mathcal{U}(N)$ is semi-regular. Due to the properness of the map $(F_1,\hdots,F_r): M \to \mathbb{R}^r$ and the regularity of $N$, 
we can identify $\mathcal{U}(N)$ with $N \times B^r$, 
where $B^r \subset \mathbb{R}^n$ is a small neighborhood of the origin $0$ in $\mathbb{R}^r$, 
and such that the level sets of $(F_1,\hdots,F_r)$ in $\mathcal{U}(N) = N \times B^r$ are $N \times \{\alpha\}, \alpha \in B^r$, and $N$ itself is identified with $N \times \{0\}$.

Repeating the above arguments for every level set $N_\alpha = N \times \{\alpha\} \subset \mathcal{U}(N) = N \times B^r$, we get a family of tori
\begin{equation}
T_\alpha = \overline{\{\exp(\sum t_i Z_i)|_{N_\alpha}\quad|\quad t_i \in \mathbb{R}\}} \subset Iso(g_\alpha),
\end{equation}
where 
\begin{equation}
Iso(g_\alpha) = \{\varphi \in \textnormal{Diffeo}(N_\alpha), \varphi_* g_{\alpha} = g_\alpha\}
\end{equation}
is the isometry group of the Riemannian metric $g_\alpha$ induced by the elliptic operator $\sum_{i=1}^p A_i + \sum_{i=1}^q Z_{j}^2|_{N_\alpha}$ on $N_\alpha.$

According to the auxilliary Theorem \ref{thm:isometry} below on the isometry groups of a family of Riemannian metrics, for each $\alpha \in B^r$ near $0$ there is an injective homomorphism
\begin{equation}
\rho_\alpha : Iso(g_\alpha) \hookrightarrow Iso(g_0).
\end{equation}
The image $\rho_\alpha(T_\alpha)$ of $T_\alpha$ under this injective homomorphism is a $l(\alpha)$-dimensional torus in $Iso(g_0)$.

Notice that $\forall \epsilon > 0$, there is a positive number $K(\epsilon)>0$ such that the set 
\begin{equation}
\{\exp(\sum_{i=1}^qt_iZ_i)|_N \quad |\quad t_1,\hdots,t_q \in [-K(\epsilon),K(\epsilon)]\}
\end{equation}
is $\epsilon$-dense in $T_0$, i.e. $\forall \psi \in T_0$ there is $\varphi = \exp(\sum_{i=1}^qt_iZ_i)|_N$ for some $t_1,\hdots,t_q \in [-K(\epsilon),K(\epsilon)]$ such that $d(\varphi,\psi) \leq \epsilon$ with respect to the distance
\begin{equation}
d(\varphi,\psi) := \max_{x\in N}d_{g_0}(\varphi(x),\psi(x)),
\end{equation}
with $d_{g_0}$ being the distance on $N$ generated by the Riemannian metric $g_0$.

If $B^r$ is small enough then $\forall \alpha \in B^r$ we have that $\exp(\sum_{i=1}^q t_iZ_i)|_N$ and $\exp(\sum_{i=1}^q t_iZ_i)|_{N\times\{\alpha\}}$ are also $\epsilon$-close for any $t_1,\hdots,t_q \in [-K(\epsilon),K(\epsilon)]$ after projecting $N\times\{\alpha\}$ to $N$ by the natural projection. Hence $\exp(\sum t_iZ_i)|_N$ and $\rho_\alpha(\exp(\sum_{i=1}^qt_iZ_i)|_{N\times\{\alpha\}})$ are also $\epsilon'$-close. Via these close elements, we can construct a map from $T_0$ to $\rho_\alpha(T_\alpha)$ which is a near-homomorphism (in the sense of Grove-Karcher-Ruh \cite{Gro1974}), and which, by Grove-Karcher-Ruh theorem \cite{Gro1974}, can be approximated by a true homomorphism
\begin{equation}
\chi_\alpha: T_0 \to \rho_\alpha(T_\alpha).
\end{equation}
This homomorphism is injective. (The kernel is trivial, because it is a subgroup of the compact group $Iso(g_0)$ and contains only elements which are close to identity. No non-trivial subgroup of a compact Lie group is like that). The injectivity of $\chi_\alpha: T_0 \to \rho(T_\alpha)$ implies in particular that 
\begin{equation}
\rho_{\alpha}^{-1}\circ\chi_\alpha: T_0 \to T_\alpha
\end{equation}
is injective, and $l(\alpha) \geq l = l(0)$ $\forall \alpha \in B^r$ (provided that $B^r$ is small enough). So we get a family of $n$-dimensional tori
\begin{equation}
\hat{T}_\alpha = \rho_{\alpha}^{-1}\circ \chi_\alpha(T_0) \subset T_\alpha \subset Iso(g_\alpha).
\end{equation}
Notice the uniqueness of the construction of $\hat{T}_\alpha$ due to the commutativity of $T_\alpha$ (which leads to the rigidity of homomorphisms from $T_0$ to $T_\alpha$).

It is easy to see that the family $\hat{T_\alpha}$ is continuous with respect to $\alpha$. Indeed, by construction this family is continuous at $\alpha = 0$. Starting at another $N \times \{\beta\}$ instead of $N\times \{0\}$ where $\beta$ is sufficiently close to $0$, we get another family of tori which is continuous at $\beta$. But the tori in this latter family contain the tori of the family $\hat{T}_\alpha$, which implies that the family $\hat{T}_\alpha$ is also continuous at $\beta$.

The continuity of the family $\hat{T}_\alpha$ means that there is a $l$-dimensional torus subgroup
\begin{equation}
\mathcal{T} \subseteq \textnormal{Homeo}(\mathcal{U}(N))
\end{equation}
such that each element $\varphi \in \mathcal{T}$ preserves every level set $N_\alpha = N \times \{\alpha\} \subseteq \mathcal{U}(N)$ and 
\begin{equation}
\varphi|_{N_\alpha} \in \hat{T}_\alpha \quad \forall \alpha \in B^r.
\end{equation}
This torus $\mathcal{T}$ is the torus action that we are looking for. The fact that $\mathcal{T}$ preserves the system is an immediate consequence of Theorem \ref{thm:action} (it is enough to look at the torus action on each level set $N_\alpha = N \times \{\alpha\}$).
\end{proof}
\begin{remark} We could prove that $\mathcal{T}$ is smooth on each level set $N_\alpha$ and is continuous in $\mathcal{U}(N)$, but we could not yet prove that $\mathcal{T}$ is smooth in $\mathcal{U}(N)$, though we suspect that this is true as well. (The proof of the smoothess of the torus action $\mathcal{T}$ in $\mathcal{U}(N)$ is a tricky problem which probably requires some very  subtle topological arguments). Nevertheless, the smoothess of $\mathcal{T}$ on every invariant level set $N_\alpha$ is good enough for doing reduction. 
If we start at two different regular level sets $N_1$ and $N_2$, then we get two Liouville torus actions which may act in a same domain but may have different dimensions (both greater of equal to $q$).
\end{remark}

\begin{thm}\label{thm:isometry}
Let $g_\alpha (\alpha \in \mathbb{R}^k)$ be a smooth $k$-dimensional family of Riemannian metrics on a smooth compact manifold $N$. Denote by 
\begin{equation} 
\mbox{Iso}(\alpha) = \{\varphi \in \mbox{Diffeo}(N) | \varphi_*g_\alpha = g_\alpha\}
\end{equation}
the isometry group of $g_\alpha$. Then there is a small neighborhood $B$ of $0$ in $\mathbb{R}^k$ such that for each $\alpha \in B$ there exists an injective group homomorphism 
\begin{equation}
\rho_\alpha: \mbox{Iso}(\alpha) \hookrightarrow \mbox{Iso}(0)
\end{equation}
such that $$\max_{\varphi \in Iso(\alpha)}\mbox{d}(\varphi,\rho_\alpha(\varphi)) \xrightarrow[\alpha \to 0]{} 0$$ 
where $\mbox{d}(\varphi,\rho_\alpha(\varphi)) = \max_{x\in N} \mbox{d}_0(\varphi(x), \rho_\alpha(\varphi)(x)),$ with $d_0$ being the distance on $N$ generated by $g_0$.
\end{thm}
\begin{proof}
First let us remark that for any $\epsilon > 0$ there exists a small neighborhood $B_\epsilon$ of $0$ in $\mathbb{R}^k$ such that $\forall \alpha \in B_\epsilon$ and $\forall \varphi \in Iso(\alpha)$ there exists $\varphi' \in Iso(0)$ such that $\varphi$ is $\epsilon$-close to $\varphi'$ with respect to the above distance, i.e.$\max_{x\in N} d_0(\varphi(x),\varphi'(x)) \leq \epsilon.$
Indeed, if it is not the case, then there exists a number $\epsilon > 0$, a family $\alpha_n \in \mathbb{R}^k$ $ (n \in \mathbb{N})$ such that $\alpha_n \xrightarrow{n\to\infty} 0$, an element $\varphi_n \in Iso(\alpha_n)$ for each $n \in \mathbb{N}$ such that $\varphi_n$ is not $\epsilon$-close to any element of $Iso(0)$. Due to the compactness of $N$, there is an infinite subsequence $(i_n) \subseteq \mathbb{N}$ and a point $x_0 \in N$ such that$\varphi_{i_n}(x_0) \xrightarrow[n\to \infty]{} y  \in N$ for some $y\in N$, and moreover, the differential $D \varphi_{i_n}(x)$ also converges when $n\to \infty$. One then deduces easily from the fact that $g_{\alpha_n} \xrightarrow{n\to\infty}g_0$ and the rigidity of isometries that $\varphi_{i_n}$ converges to some element $\varphi \in Iso(\alpha)$ when $n \to \infty$. But it means that $\varphi_{i_n}$ is $\epsilon$-close to $\varphi \in Iso(\alpha)$ when $n$ is big enough, which is a contradiction.

Thus, for every small number $\epsilon > 0$, if $\alpha \in \mathbb{R}^k$ is close enough to $0$ in $\mathbb{R}^k$ then we can construct a map
\begin{equation}
\mu_\alpha: Iso(\alpha) \to Iso(0)
\end{equation}
such that $\mu_\alpha(\varphi)$ is $\epsilon$-close to $\varphi$ for all $\varphi \in Iso(\alpha)$. We can arrange so that this map $\mu_\alpha$ is at least piecewise-continuous (for a partition of $Iso(\alpha)$ into a finite number of closed domains). It is then clear that $\mu_\alpha$ is a near-homomorphism, i.e. $\mu_\alpha(\varphi\circ\psi)$ is $\epsilon$-close to $\mu_\alpha(\varphi).\mu_\alpha(\psi)$ for any $\varphi, \psi \in Iso(\alpha)$, where $\epsilon'$ is a small positive number depending on $\epsilon$ but does not depend on $\varphi$ and $\psi$ ($\epsilon' \to 0$ when $\epsilon \to 0$). According to a classical result of Grove, Karcher and Ruh \cite{Gro1974}, any such near-homomorphism between two compact Lie groups and be approximated by a true homomorphism. Thus we get a homomorphism
\begin{equation}
\rho_\alpha: Iso(\alpha) \to Iso(\alpha)
\end{equation}
which is close to $\mu_\alpha$, i.e. $d(\rho_\alpha(\phi),\mu_\alpha(\varphi)) \leq \epsilon'' \quad \forall \varphi \in Iso(\alpha)$ (for some $\epsilon''$ such that $\epsilon'' \to 0$ when $\epsilon \to 0$).

But it also means that $d(\rho_\alpha(\varphi),\varphi) \leq \epsilon'''$
where $\epsilon''' \to 0$ when $\epsilon \to 0$. The last inequality also implies that $\rho_\alpha$ must be injective when $\alpha$ is close enough to $0$.
\end{proof}
\subsection{Normal form for integrable SDS of types $(0,q,r)$ and $(1,q,r)$}

According to Definition \ref{defn:pqr}, an SDS $(M,X)$ is called integrable of type $(0,q,r)$ if there exists an deterministic integrable system $(Z_1,\hdots, Z_q, F_1,\hdots,F_r)$ on $M$ (where $q+r = \dim M$)  such that each $Z_i$ is an infinitesimal diffusion symmetry of $X$ and each $F_i$ is a strong first integral of $X.$ We will assume that the map $(F_1,\hdots,F_r): M \to \mathbb{R}^r$ is proper. According to Theorem \ref{thm:Liou} and Theorem \ref{thm:316}, each connected regular level set of the system is a $q$-dimensional torus in a neighborhood of which we have a free $\mathbb{T}^q$-action which preserves $X$ and $Z_1,\hdots,Z_q,F_1,\hdots,F_r.$ As a consequence, we have the following normal form theorem:

\begin{thm}[Normal form for SDS's of type $(0,q,r)$]\label{thm:0qr}
Let $X = X_0 + \sum_{i=1}^m X_i \circ \dfrac{dB_{t}^i}{dt}$ be an integrable SDS of type (0,q,r), with the aid of a deterministic integrable system $(Z_1,\hdots,Z_q,F_1,\hdots,F_r)$ on a manifold $M$. Assume that the map $(F_1,\hdots,F_r): M \to \mathbb{R}^r$ is proper and that  $N= \{F_1=c_1,\hdots,F_r=c_r\}$ is a connected regular level set of the system, i.e. $dF_1(x)\wedge\hdots\wedge dF_r(x) \ne 0$ and $Z_1(x)\wedge\hdots\wedge Z_q(x) \ne 0$ for every point $x\in N$. Then $N$ is diffeomorphic to a torus $\mathbb{T}^q$, and in a tubular neighborhood $\mathcal{U} (N) \cong \mathbb{T}^q \times B^r$ of $N$ there is a coordinate system $(\theta_1(mod 1),\hdots,\theta_q(mod 1), \gamma_1,\hdots\gamma_r)$ such that $X$ is diffusion equivalent to a system $Y$ of the type
\begin{equation}
Y = Y_0 + \sum_{i=1}^m Y_i \circ \dfrac{dB_{t}^i}{dt},
\end{equation} where 
\begin{equation}Y_i  = \sum_{k=1}^q a_{ik}(\gamma_1,\hdots,\gamma_r)\partial_{\theta_k}\quad (i=0,1,\hdots,m)\end{equation} are vector fields which are tangent to the tori $\mathbb{T}^q \times \{pt\}$ and whose coefficients are constant on the tori.
\end{thm}
\begin{proof}
It follows directly from Definition \ref{defn:pqr} and Theorem \ref{thm:actiondiff} that $N \cong \mathbb{T}^q$ and $A_X$ is invariant with respect to the Liouville $\mathbb{T}^q$-action . Let $(\theta_1(mod 1),\hdots,\theta_q(mod 1),\gamma_1,\hdots,\gamma_r)$ be a coordinate system in $\mathcal{U} \cong \mathbb{T}^q \times B^r$ compactible with the Liouville $\mathbb{T}^q$-action, i.e. the action is by translations in the periodic  coordinates $\gamma_1,\hdots,\gamma_r$.

For $i=1,\hdots,m$ denote by $Y_i$ the vector field
\begin{equation}
Y_i(\theta_1,\hdots,\theta_q,\gamma_1,\hdots,\gamma_r) = X_i(0,\hdots.0,\gamma_1,\hdots,\gamma_r),
\end{equation}
i.e. $Y_i$ is $\mathbb{T}^q-$invariant and coincides with $X_i$ at the section $\{\theta_1=0,\hdots,\theta_q=0\}$. Then $\dfrac{1}{2}\sum Y_{i}^2$ and $A_X$ have the same principal symbol at  $\{\theta_1=0,\hdots,\theta_q=0\}.$ But since both $\dfrac{1}{2}\sum Y_{i}^2$ and $A_X$ are $\mathbb{T}^q-$invariant, they have the same principal symbol everywhere. Thus $A_X - \dfrac{1}{2}\sum Y_{i}^2$ is a first order operator which is $\mathbb{T}^q-$invariant, i.e. a $\mathbb{T}^q$-invariant vector field. Put $Y_0 = A_X - \dfrac{1}{2}\sum Y_{i}^2,$ we have that $Y = Y_0 + \sum_{i=1}^m Y_i \circ \dfrac{dB_{t}^i}{dt}$ is $\mathbb{T}^q$-invariant and is diffusion equivalent to $X$. The fact that $X$ is tangent to the level sets implies that the vector fields $Y_0,\hdots,Y_m$ are of the form $Y_i = \sum_{k=1}^q a_{ik}(\gamma_1,\hdots,\gamma_r)\partial_{\theta_k}.$
\end{proof}

Consider now an integrable SDS $(A_X,Z_1,\hdots,Z_q,F_1,\hdots,F_r)$ of type $(1,q,r),$ where $A_X = X_0 + \dfrac{1}{2}\sum_{i=1}^m X_i^2$ is the diffusion generator of $X=X_0 + \sum X_i \circ \dfrac{dB_{t}^i}{dt}$. As before, we will assume that $N$ is a connected compact regular level set of the system. According to Theorem \ref{thm:316}, there is either an effective $\mathbb{T}^{q+1}$-action or an effective $\mathbb{T}^q$-action which preserves the system in a neighborhood of $N$. In the case of a $\mathbb{T}^{q+1}$-action, $X$ is in fact integrable with the aid of a system $(\Theta_1,\hdots,\Theta_{q+1},F_1,\hdots,F_r)$ of type $(0,q+1,r)$, where $\Theta_1,\hdots,\Theta_{q+1}$ are generators of the $\mathbb{T}^{q+1}$-action. Consider now the case when the torus action is really of dimension $q$ and not $q+1$. In this case, semi-locally in the neighborhood of a regular level set $N$, we can replace $Z_1,\hdots,Z_q$ by $q$ generators $\Theta_1,\hdots,\Theta_{q}$ of the $\mathbb{T}^q$-action. Making the reduction of the system with respect to this $\mathbb{T}^q$-action, we get a system of type $(1,0,r)$, i.e. a $r$-dimensional family of $1$-dimensional SDS.

Remark also that, due to the fact that $\dim N = q+1$ and there is an effective $\mathbb{T}^q$-action on $N$ (which is free almost everywhere), it is easy to classify $N$ topologically: either $N$ is a $(q+1)$-dimensional torus, or $N$ is a ``lense space'', which is obtained by gluing together $2$ copies of $\mathbb{T}^{q-1} \times D^2$ (where $D^2$ is a $2$-dimensional disk) along the boundary $\mathbb{T}^{q-1} \times \partial D^2 \cong \mathbb{T}^q$ via some automorphism of $\mathbb{T}^q$.

Similarly to the $(0,q,r)$-type case, in the $(1,q,r)$-type case we can also have a normal form for the system near a regular orbit of the Liouville torus action. More generally, we have the following simple result, whose proof is absolutely similar to the proof of Theorem \ref{thm:0qr}:
\begin{thm}\label{thm:1qr}
Assume that a SDS $X = X_0 + \sum_{i=1}^m X_i \circ \dfrac{dB_{t}^i}{dt}$ is diffusion invariant with respect to an effective action of a torus $\mathbb{T}^l$ on a manifold $M$. Let $K$ be a regular orbit of this torus action in $M$. Then $K \cong \mathbb{T}^l$, and in a neighborhood of $K$ the system $X$ is diffusion equivalent to a system $Y=Y_0 + \sum_{i=1}^m Y_i \circ \dfrac{dB_{t}^i}{dt}$ such that $Y_0,Y_1,\hdots,Y_m$ are invariant with respect to the $\mathbb{T}^l$-action.
\end{thm}

\subsection{Integrable SDS's of type (p,0,0)} 
Given an arbitrary integrable SDS of type $(p,q,r)$, one can reduce it with respect to the Liouville torus $\mathbb{T}^l$-action provided by Theorem \ref{thm:316}, with $l\geq q$. After the reduction, the system becomes a system of type $(p',0,r)$, where $p' = p+q-l \leq p$. If we restrict it to the level sets, then after reduction and restriction it becomes an integrable system of type $(p',0,0)$.

Consider now an integrable SDS of type $(p,0,0)$. It means a $p$-tuple of diffusion generators $A_1,\hdots, A_p$ on a $p$-dimensional manifold $M^p$ which commute pairwise. We will assume that for almost every $x \in M$, the restriction of the symbols $\hat{A_i} : T^* M \to \mathbb{R}$ to $T_{x}^* M$ is a linear independent family of quadratic functions on $T_{x}^* M$. This additional assumption is put here in order to avoid the ``fake'' $(p,0,0)$ type, for example given by a family of the type $(A_1, A_2= F_1A_1,\hdots, A_p = F_{p-1}A_1)$ where $F_1,\hdots, F_{p-1}$ are functions which commute with $A_1$. (It would be more natural to consider this example as a system of type $(1,0,p-1)$ given by the family $(A_1, F_1,\hdots, F_{p-1})$ than a system $(A_1,F_1A_1,\hdots, F_{p-1}A_1)$ of type $(p,0,0)$).

\begin{prop}
With the above notations and assumptions, at any point $x \in M$ where the symbols $\hat{A}_1|_{T_{x}^* M}, \hat{A}_2|_{T_{x}^* M},\hdots, \hat{A}_p|_{T_{x}^* M}$ are functionally independent as a functions on $T^*M$, and for any constants $\alpha_1,\hdots, \alpha_p > 0$ the operator $\sum \alpha_i A_i$ is elliptic at $x$, i.e the sum $\sum \alpha_i \hat{A}_i|_{T_{x}^* M} : T_{x}^* M \to \mathbb{R}$ is a positive definite quadratic form.
 \end{prop}

\begin{proof}
For each point $x \in M$ and a diffusion operator $A$ on $M$, denote by $\mbox{span}_x \hat{A} \subset T_{x}^* M$ the tangent vector subspace of $T_{x} M$ spanned by $\hat{A}$ at $x$. In other words, $Z \in \mbox{span}_x \hat{A}$ if and only if $\langle z,\alpha\rangle = 0$ for any $\alpha \in T_{x}^* M$ such that $\hat{A}(\alpha) = 0$.

If $\hat{A} = \dfrac{1}{2} \sum_{i=1}^k \hat{X}_{i}^2$ where $X_i$ are vector fields then \begin{equation}\label{eq:span}\mbox{span}_x \hat{A} = \mbox{Vect}\{ X_1(x),\hdots, X_k(x)\}\end{equation} though of course the definition of $\mbox{span}_x \hat{A}$ is intrinsic and does not depend on the choice of $X_i$. It is also easy to see that if $\alpha_1,\hdots, \alpha_k > 0$ then $\mbox{span}_x(\sum \alpha_i \hat{A}_i)$ is the sum of the vector spaces $\mbox{span}_x (A_i)$

The ellipticity of $\sum \alpha_i A_i$ at $x$ means that $\sum \alpha_i \hat{A}_i |_{T^* M}$ is definite positive, i.e. \begin{equation}\mbox{span}_x(\sum \alpha_i \hat{A}_i) = T_x M.\end{equation}

If $x \in M$ is a point such that $\mbox{span}_x(\sum \alpha_j \hat{A}_j) \ne T_x M$, then there is a function $f : M \to \mathbb{R}$ such that $df(x) \ne 0$ but $df(x)$ vanishes on $\mbox{span}_x(\sum \alpha_i \hat{A}_i)$, i.e $\langle df(x), Y\rangle$ for any $j$ and any $Y \in \mbox{span}_x(\hat{A}_j)$. It implies that the Poisson bracket $\{\hat{A}_i ,F\} = 0$ on $T_{x}^* M$. But $\{\hat{A}_i, \hat{A}_j\} = 0$ due to our commutativity assumption. It follows that $d\hat{A}_1,\hdots, d\hat{A}_1, dF$ are linearly dependent at every point of $T_{x}^* M$, i.e. $d\hat{A}_1 \wedge \hdots \wedge d\hat{A}_p \wedge dF$ vanishes at every point of $T^* M$. The reason is that we cannot have more than $p$ linearly independent vector $Y_1, Y_2,\hdots, Y_p$ in a symplectic space of dimension $2p$ such that $\omega(Y_i,Y_j) = 0$ for all $i,j$. It implies that the restriction of $A_1,\hdots, A_p$ to $T_{x}^* M$ are functionally dependent. So $x$ is not a generic point of $M$, by our assumptions. At a generic point $x \in M$ we will have that $\mbox{span}_x (\sum \alpha_i \hat{A}_i) = T_xM$, i.e. $\sum \alpha_i A_i$ is elliptic at $x$.
\end{proof}

The family of principal symbols
\begin{equation}
H_{\alpha_1,\hdots,\alpha_p}:= \sum_{i=1}^p \alpha_i \hat{A}_i : T^*M \to \mathbb{R}
\end{equation}
$(\alpha_1,\hdots,\alpha_p > 0)$ is a commuting family of homogenous positive definite quadratic Hamiltonian functions on $T^*M$, whose corresponding Hamiltonian vector fields are the geodesic flows of a family of Riemannian metrics on $M$. Thus we get a $p$-dimensional family of Riemannian metrics on $M$, whose geodesic flows are integrable and commute with each other.

The problem of finding integrable geodesic flows and quantizing them into integrable diffusion operators is a big problem in geometry, which is out of the scope of this paper. Let us just mention that the so called \textbf{projectively equivalent metrics} (i.e. families of different metrics having the same un-parametrized geodesics) have been studied intensively by Matveev, Bolsinov, Topalov \cite{BoMa2003,Mat2000} and other authors. In particular, they showed that these metrics are integrable, can be quantized into integrable SDS of type $(p,0,0)$ in our sense, and are essentially the same as the so called (separable St$\ddot{a}$ckel-) Benenti systems studied by Benenti and many other authors in classical and quantum mechanics, see, e.g., \cite{Blasz2013,BoMa2003,DuVa2005,Mat2000}. The metrics (induced from $\mathbb{R}^n$) on multi-dimensional ellipsoids belong to this family of integrable metrics, and so the Brownian motion on a $p$-dimensional ellipsoid is an integrable SDS of type $(p,0,0)$.

\subsection{Reduced integrability of SDS's}
In Hamiltonian dynamics, when one talks about the integrability of a Hamiltonian system, one often actually means its \textbf{reduced integrability}, i.e. the integrability not of the original system, but of the reduced system with respect to some proper group action. For example, the famous Kovalevskaya top is originally a Hamiltonian system on $T^*SO(3)$ with 3 degrees of freedom, but people often consider it as a (reduced) integrable system with 2 degrees of freedom, see, e.g. the book by Bolsinov and Fomenko \cite{BoFo2004}. See \cite{Jova2002,Zung2006} for some theorems on the relation integrability and reduced integrability. Concerning SDS's, we can formulate the following conjecture:

\begin{conj}
Let $X = X_0 + \sum X_i \circ \dfrac{dB_t^i}{dt}$ be an SDS on a manifold $M$ which integrable and is diffusion invariant with respect to an action of a compact Lie group $G$ on $M$. Then the reduced system of $X$ on the (regular part of the) quotient space $M\slash G$ is also an integrable SDS.
\end{conj}


\begin{thebibliography}{10}
\baselineskip0.4cm
\parskip-0.1cm
\small{
\bibitem{AF1995}Albeverio, S. and Fei, S. \textit{Remark on symmetry of stochastic dynamical systems and their conserved quantities} J. Phys. A, 28 (1995), pp. 6363–6371.
\bibitem{Blasz2013}Błaszak, M; Domański, Z; Sergyeyev, A; Szablikowski, B. \textit{Integrable quantum Stäckel systems},Phys. Lett. A 377 (2013), no. 38, 2564–2572.
\bibitem{Bismut1981}Bismut, J.M. Mecanique Aleatoire. Lecture Notes in Mathematics, volume 866. Springer-Verlag.
\bibitem{BoMa2003} Bolsinov, A.V and Matveev, V.S. \textit{Geometrical interpretation of Benenti systems}. J. Geom. Phys. 44 (2003), no. 4, 489–506.
\bibitem{Fre1} Borodin, A. N and Freidlin, M. I. \textit{Fast oscillating random perturbations of dynamical systems with conservation laws.} Ann. Inst. H. Poincaré Probab. Statist. 31 (1995), no. 3, 485–525.
\bibitem{DuVa2005}Duval, C. and Valent, G. \textit{Quantum integrability of quadratic Killing tensors}. J. Math. Phys. 46 (2005), no. 5, 053516, 22 pp.
\bibitem{BoFo2004} Fomenko, A.T. and Bolsinov, A.V. Integrable Hamiltonian Systems: Geometry, Topology, Classification. CRC Press; 1 edition (February 25, 2004)
\bibitem{Fre2} Freidlin, M. and Weber, M. \textit{Random perturbations of dynamical systems and diffusion processes with conservation laws}. Probab. Theory Related Fields 128 (2004), no. 3, 441–466. 
\bibitem{Gal} Galmarino, A.R. \textit{Representation of an isotropic diffusion as a skew product.} Zeitschrift für Wahrscheinlichkeitstheorie und Verwandte Gebiete 1963, Volume 1, Issue 4, pp 359-378.
\bibitem{Gitterman2005}Gitterman, M. The Noisy Oscillator: The First Hundred Years, from Einstein Until Now, World Scientific Publishing Company.
\bibitem{Gro1974} Grove, K; Karcher, H; Ruh, E.A. \textit{Group actions and curvature.} Invent. Math. 23 (1974), 31–48.
\bibitem{I1989}Ikeda, N. and Watanabe, S. Stochastic Differential Equations and Diffusion Processes. Second edition. North-Holland Mathematical Library, 24. North-Holland Publishing Co.
\bibitem{Jova2002} Jovanovic,B. \textit{Symmetries and integrability.} Publ. Inst. Math. (Beograd) (N.S.) 84(98) (2008), 1–36. 
\bibitem{Ku1997}Kunita, H. Stochastic Flows and Stochastic Differential Equations,Cambridge University Press.
\bibitem{Or22007} Lazaro-Cami, J.-A. and Ortega, J.-P. \textit{Reduction, reconstruction, and skew-product decomposition of symmetric stochastic differential equations.} Stoch. Dyn. 9 (2009), no. 1, 1–46.
\bibitem{XuLi2008} Li, Xue-Mei. \textit{An averaging principle for a completely integrable stochastic Hamiltonian system}. Nonlinearity 21 (2008), no. 4, 803–822.
\bibitem{Liao2008}Liao, M. \textit{A decomposition of Markov Processes via Group Action.} Journal of Theoretical Probability, 01/2009; 22(1):164-185.
\bibitem{Liu1855} Liouville, J. \textit{Note sur l'integration des equations differentielle de la dynamique} presentee au bureau des longtitudes de 29 juin 1853, Journal de Mathematiques pures et appliquees 20 (1855) 137-138.
\bibitem{Makus1988}Markus, L. and Weerasinghe, A. \textit{Stochastic oscillators.} J. Differential Equations 71 (1988), no. 2, 288–314.
\bibitem{Mat2000} Matveev, V.S. \textit{Quantum integrability of the Beltrami–Laplace operator for geodesically equivalent metrics}, Russ. Math. Dokl. 61 (2) (2000) 216–219
\bibitem{Mi1999} Misawa, T. \textit{Conserved Quantities and Symmetries Related to Stochastic Dynamical Systems}, Annals of the Institute of Statistical Mathematics December 1999, Volume 51, Issue 4, pp 779-802.
\bibitem{Ok2003}Øksendal, B.Stochastic Differential Equations. Sixth Edition. Universitext. Springer-Verlag.
\bibitem{PoRo} Pauwels, E.J and Rogers, L.C.G. \textit{Skew-Product Decompositions of Brownian Motions}, Contemporary Math., 73 (1988).
\bibitem{LiDiffOpe} Taylor, M. Pseudodifferential Operators, Four lectures at MSRI, September 2008.
\bibitem{Zung2012} Zung, N.T. \textit{A general approach to the problem of action-angle variables}, in preparation. (See also the earlier version: Action-angle variables on Dirac manifolds, Arxiv: 1204.3865).
\bibitem{Zung2006} Zung, N.T. \textit{Torus actions and integrable systems.} Topological methods in the theory of integrable systems, 289–328, Camb. Sci. Publ., Cambridge, 2006.
\bibitem{Zung} Zung, N.T and Thien, N.T \textit{ Physics-like second-order models of financial assets prices}, in preparation.
}
\end{thebibliography}
\end{document}